\documentclass{article}

\usepackage{amssymb}
\usepackage{amsmath}
\usepackage{amsfonts}
\usepackage{amsthm}
\usepackage{algorithm}
\usepackage{a4wide}
\usepackage{hyperref}
\usepackage{enumerate}
\usepackage{mathtools}

\newtheorem{theorem}{Theorem}
\newtheorem{lemma}{Lemma}
\newtheorem{corollary}{Corollary}
\newtheorem{problem}{Problem}
\newtheorem{remark}{Remark}
\newtheorem{claim}{Claim}

\newcommand{\llbracket}{[\![}
\newcommand{\rrbracket}{]\!]}

\DeclareMathOperator{\ord}{ord}
\DeclareMathOperator{\act}{act}
\DeclareMathOperator{\xor}{xor}

\begin{document}

\title{On the Subgroup Distance Problem in Cyclic Permutation Groups}

\author{
  Andreas Rosowski\\
  \text{rosowski@eti.uni-siegen.de}
}
\date{
University of Siegen, Germany
}

\maketitle
\begin{abstract}
We show that the \textsc{Subgroup distance problem} regarding the Hamming distance, the Cayley distance and the $l_\infty$ distance is {\sf NP}-complete when the input group is cyclic. When we restrict the $l_\infty$ distance to fixed values we show that it is {\sf NP}-complete to decide whether there are numbers $z_1,z_2 \in \mathbb{N}$ such that $l_\infty(\beta, \alpha_1^{z_1}\alpha_2^{z_2}) \leq 1$ for permutation $\alpha_1,\alpha_2,\beta \in S_n$ where $\alpha_1$ and $\alpha_2$ commute. However on the positive side we can show that it can be decided in {\sf NL} whether there is a number $z \in \mathbb{N}$ such that $l_\infty(\beta, \alpha^z) \leq 1$ for permutations $\alpha,\beta \in S_n$. For the former we provide a tool, namely for all numbers $t_1,t_2,t \in \mathbb{N}$ where $t$ is required to be odd, $0 \leq t_1 < t_2 < t$ and $t_1 \not\equiv t_2 \bmod q$ for all primes $q \mid t$ we give a constructive proof for the existence of permutations $\alpha,\beta \in S_t$ with $l_\infty(\beta, \alpha^{t_1}) \leq 1$ and $l_\infty(\beta, \alpha^{t_2}) \leq 1$.
\end{abstract}

\section{Introduction}
Bijective functions on a set $\Omega$ are called permutations. The set of all permutations on $\Omega$ forms a group $\mathsf{Sym}(\Omega)$, the so called symmetric group on $\Omega$. The group operator is the composition of functions. Subgroups of $\mathsf{Sym}(\Omega)$ are also called permutation groups. We only consider finite permutation groups. With $S_n$ we denote the symmetric group where $\Omega = \{1,\dots,n\}$. The order of a subgroup of $S_n$, i.e. the number of elements of this group, can be exponentially large in $n$. For instance $S_n$ contains $n!$ permutations. Therefore permutation groups are usually given by a set of generators. And in fact for $n > 3$ every subgroup of $S_n$ can be generated by a generating set of size at most $\frac{n}{2}$ \cite{lucchini} and thereby provides a much more succinct representation. In such a setting where the group elements are no longer given explicitly it is a priori not clear how efficient subgroup membership checking can be done. However it was shown that subgroup membership checking can be done in polynomial time when the permutation group is given by a set of generators \cite{furst,sims}. Later it was shown that it can even be done in {\sf NC} by \cite{babai}. There are many more algorithmic problems that can be solved in polynomial time when the permutation group is given by a set of generators \cite[Chapter~3]{seress}.

Even in the case that a given permutation is not a member of a group $G$ one might still ask how close this permutation is to $G$. This leads us to the following problem that we study:
\begin{problem}[\textsc{Subgroup distance problem}]\label{problemSubgroupDistance}~\\
Input: $\gamma_1,\dots,\gamma_m,\gamma \in S_n, k \in \mathbb{N}$.\\
Question: Is there an element $\delta \in \langle \gamma_1,\dots,\gamma_m \rangle$ such that $d(\gamma,\delta) \leq k$?
\end{problem}
Here $d$ is a metric on $S_n$. Note that the unary encoded number $n$ is part of the input. For evaluation $\pi(i)$ of a permutation $\pi \in S_n$ at position $i \in \{1,\dots,n\}$ we use the notation $i^\pi$. We investigate the \textsc{Subgroup distance problem} with respect to the following metrics:
\begin{itemize}
\item The Hamming distance of two permutations $\tau, \pi \in S_n$ is defined as
\begin{displaymath}
H(\tau, \pi) = |\{i \mid i^\tau \neq i^\pi\}|.
\end{displaymath}
\item The Cayley distance of two permutations $\tau, \pi \in S_n$ is defined as
\begin{displaymath}
C(\tau,\pi) = \text{minimum number of transpositions taking } \tau \text{ to } \pi.
\end{displaymath}
By \cite{diaconis} this can be expressed as
\begin{displaymath}
C(\tau,\pi) = n - \text{number of cycles in } \tau\pi^{-1}
\end{displaymath}
where fixed-points also count as cycles. We will always use the second expression.
\item The $l_\infty$ distance of two permutations $\tau, \pi \in S_n$ is defined as
\begin{displaymath}
l_\infty(\tau,\pi) = \max_{1 \leq i \leq n} |i^\tau - i^\pi|.
\end{displaymath}
\end{itemize}
Our main result is that the \textsc{Subgroup distance problem} regarding these metrics is already {\sf NP}-complete when the input permutation group is cyclic. In an extended arXiv version \cite{rosowski} of this paper we prove analogous results also for the $l_p$ distance (for all $p \geq 1$), the Lee distance, Kendall's tau distance and Ulam's distance. We did not include these results in the present paper in order to keep it clearer and since the proofs are based on ideas similiar to those in this paper. The paper \cite{deza} is a good survey about metrics and their applications, see also \cite[Chapter~6]{diaconis} for more information about these metrics.

Moreover we investigate the \textsc{Subgroup distance problem} regarding the $l_\infty$ distance in the case when $k$ from Problem~\ref{problemSubgroupDistance} is a fixed constant. For $k=1$ we show that the \textsc{Subgroup distance problem} is {\sf NP}-complete when the input group is abelian and given by at least two generators and can be solved in non-deterministic logspace ({\sf NL} for short) when the input group is given by a single generator.

Our motivating results for this are from \cite{buchheim} where it was shown that the \textsc{Subgroup distance problem} regarding all metrics mentioned above is {\sf NP}-complete when the input group is abelian of exponent $2$ and from \cite{onur} where it was shown that the \textsc{Subgroup distance problem} has applications in cryptography. We also would like to mention that the \textsc{Subgroup distance problem} regarding the Cayley distance was already shown to be {\sf NP}-complete when the input group is abelian of exponent $2$ by \cite{pinch}. When considering the \textsc{Subgroup distance problem} in the case $k=0$ this problem simply becomes a subgroup membership problem for permutation groups which can be solved in polynomial time by the Schreier-Sims algorithm \cite{furst,sims} and was later shown to be solvable in {\sf NC} \cite{babai}.

\subsection{Related Work}
In \cite{buchheim} also the maximum subgroup distance problem was studied where for given permutations $\pi_1,\dots,\pi_m,\tau \in S_n$ and $k \in \mathbb{N}$ it is asked whether there is an element $\pi \in \langle \pi_1,\dots,\pi_m \rangle$ such that $d(\tau,\pi) \geq k$? This problem has also been shown to be {\sf NP}-complete when the input group is abelian of exponent $2$ regarding all metrics mentioned in the introduction except for the $l_\infty$ metric. In this case the problem can be solved in polynomial time.

In \cite{cameron} the weight problem and variants were studied. The weight of a permutation $\pi \in S_n$ with respect to some metric $d$ is defined as $w_d(\pi) = d(\pi,\mathrm{id})$ and the question is whether for given permutations $\pi_1,\dots,\pi_m \in S_n$ and $k \in \mathbb{N}$ there is $\pi \in \langle \pi_1,\dots,\pi_m \rangle$ such that $w_d(\pi) = k$? In the maximum weight problem it is instead asked whether there is $\pi \in \langle \pi_1,\dots,\pi_m \rangle$ such that $w_d(\pi) \geq k$? The minimum weight problems asks whether there is $\pi \in \langle \pi_1,\dots,\pi_m \rangle \setminus \{\mathrm{id}\}$ such that $w_d(\pi) \leq k$? These problems regarding several metrics were shown to be {\sf NP}-complete except for the maximum weight problem regarding the $l_\infty$ metric which has been shown to be solvable in polynomial time. Note that the {\sf NP}-completeness of the weight problem regarding the Hamming metric was already shown in \cite{buchheimjunger}.

In \cite{arvind} the computational complexity of the minimum weight problem and the subgroup distance problem was studied in a deterministic setting regarding exact and approximation versions.

In \cite{arvind2} the parameterized complexity of the maximum weight problem regarding the Hamming metric was studied.

\section{Preliminaries}
We will occasionally need the following lemma that seems to be folklore:
\begin{lemma}\label{lemmasplit}
Let $\alpha \in S_n$ be a cycle of length $l \leq n$. Then $\alpha^x$ splits into $\gcd(x,l)$ many disjoint cycles of length $\frac{l}{\gcd(x,l)}$.
\end{lemma}
A proof can be found in \cite{lohrey}. All proofs of {\sf NP}-hardness will start from one of the following problems:
\begin{problem}[\textsf{3-SAT}]~\\
Input: a finite set $X$ of variables and a set $C$ of clauses over $X$ with $|c| = 3$ for all $c \in C$.\\
Question: Is there a satisfying truth assignment for $C$?
\end{problem}
\begin{problem}[\textsf{X3HS}]~\\
Input: a finite set $X$ and a set $\mathcal{B} \subseteq 2^X$ of subsets of $X$ all of size $3$.\\
Question: Is there a subset $X' \subseteq X$ such that $|X' \cap C| = 1$ for all $C \in \mathcal{B}$?
\end{problem}
All of these problems are {\sf NP}-complete \cite{gareyjohnson}. For this also note that \textsf{X3HS} is the same problem as positive \textsf{1-in-3-SAT}.

\subsection{Permutations}
We denote with $S_n$ the set of all permutations on the set $\{1,\dots,n\}$ for some integer $n \geq 1$. By $\mathrm{id}$ we denote the permutation that fixes all points. For a permutation $\pi \in S_n$ and all $i \in \{1,\dots,n\}$ we use $i^\pi$ to denote the unique $j \in \{1,\dots,n\}$ such that $\pi(i) = j$. Moreover we evaluate from left to right, i.e. for permutations $\pi_1,\dots,\pi_m \in S_n$ and some $a_0,a_1,\dots,a_m \in \{1,\dots,n\}$ we have $a_0^{\pi_1 \cdots \pi_m} = a_m$ if and only if for $i=1,\dots,m-1$ we have $a_{i-1}^{\pi_i \cdots \pi_m} = a_i^{\pi_{i+1} \cdots \pi_m}$ and $a_{m-1}^{\pi_m} = a_m$.

We assume that permutations are given in standard representation. There are two standard representations: the pointwise representation where a permutation $\pi \in S_n$ is represented by a list $[1^\pi,2^\pi,\dots,n^\pi]$ and the cycle representation where $\pi$ is represented by a list of its pairwise disjoint cycles. Fixed-points are usually not included in this list. The standard representations can be transformed into each other in log-space \cite{cook}.

\subsection{Notations}
For a cycle $\gamma$ we define
\begin{displaymath}
\act(\gamma) = \begin{cases}
\{i\} & \text{if } \gamma \text{ is a 1-cycle identifying the fixed-point } i^\gamma = i\\
\{i \mid i^\gamma \neq i\} & \text{if } \gamma \text{ has length at least }2.
\end{cases}
\end{displaymath}
By $\ord(\alpha)$ where $\alpha \in S_n$ we denote the order of $\alpha$ i.e. the smallest non-negative integer $i \geq 1$ such that $\alpha^i = \mathrm{id}$. With $\nu_p(n)$ we denote the $p$-adic valuation of the integer $n \in \mathbb{Z}$, i.e. the largest positive integer $d$ such that $n \equiv 0 \bmod p^d$. We use the notation $[i,j]$ to denote the set $\{i,i+1,i+2,\dots,j\}$ for integers $i \leq j$. Moreover we use $\llbracket i,j \rrbracket$ to denote the cycle $(i,i+1,i+2,\dots,j) \in S_n$ for non-negative integers $1 \leq i < j \leq n$. We also use $\llbracket i \rrbracket$ instead of $\llbracket 1,i \rrbracket$ for a non-negative integer $2 \leq i \leq n$.

\section{Subgroup Distance Problem}\label{sectionSubgroupDistanceProblem}
In the following sections when we show {\sf NP}-completeness results we only show the hardness since membership in {\sf NP} has already been shown in \cite{buchheim} for all metrics from the introduction.

\subsection{Hamming Distance}
\begin{lemma}\label{lemmahammingdistance}
Let $l \geq 2$ and $0 \leq e \leq l-1$ be integers. Then $\llbracket l \rrbracket^x$ and $\llbracket l \rrbracket^e$ match at $l$ positions if $x \equiv e \bmod l$ and mismatch at $l$ positions if $x \not \equiv e \bmod l$.
\end{lemma}
\begin{proof}
Let $1 \leq i \leq l$ and let $0 \leq y \leq l-1$ be such that $y \equiv x \bmod l$. Then we have
\begin{displaymath}
i^{\llbracket l \rrbracket^x} = \begin{cases}
i+y & \text{if } i+y \leq l\\
i+y-l & \text{otherwise}
\end{cases}
\end{displaymath}
and
\begin{displaymath}
i^{\llbracket l \rrbracket^e} = \begin{cases}
i+e & \text{if } i+e \leq l\\
i+e-l & \text{otherwise}.
\end{cases}
\end{displaymath}
Therefore we have $i^{\llbracket l \rrbracket^x} = i^{\llbracket l \rrbracket^e}$ if and only if $i+y = i+e$ or $i+y-l = i+e-l$ if and only if $y = e$ if and only if $x \equiv e \bmod l$. Note that the cases $i+y = i+e-l$ and $i+y-l = i+e$ cannot occur since we would get $y = e-l < 0$ and $y = e+l > l-1$ which contradict $0 \leq y \leq l-1$.
\end{proof}

\begin{theorem}
The {\sc Subgroup distance problem} regarding the Hamming distance is {\sf NP}-complete when the input group is cyclic.
\end{theorem}
\begin{proof}
We give a log-space reduction from \textsc{3-SAT}. Let $X = \{x_1,\dots,x_n\}$ be a set of variables and let $C = \{c_1,\dots,c_m\}$ be a set of clauses over $X$ where $c_j$ contains exactly $3$ different literals for all $j \in [1,m]$. W.l.o.g. we can assume that no clause contains a positive and a negative literal regarding the same variable. For $j \in [1,m]$ we define $I_j \subseteq [1,n]$ as the set of all indices $i$ such that $c_j \cap \{x_i,\bar{x}_i\} \neq \emptyset$. Let $p_1,\dots,p_n$ be the first $n$ odd primes. Moreover let $q_j = \prod_{i \in I_j} p_i$ for $j=j \in [1,m]$ and let $N = 2\sum_{i=1}^n p_i + 7\sum_{j=1}^m q_j$. We will work with the group $G \leq S_N$ in which
\begin{displaymath}
G = \prod_{i=1}^n V_i \times \prod_{j=1}^m U_j
\end{displaymath}
with $V_i = S_{p_i}^2$ and $U_j = S_{q_j}^7$. We define the input group elements as
\begin{align*}
\tau&= (\alpha_1,\dots,\alpha_n, \beta_1,\dots,\beta_m)\\
\pi&= (\gamma_1,\dots,\gamma_n,\delta_1,\dots,\delta_m)
\end{align*}
with $\alpha_i = (\llbracket p_i \rrbracket, \mathrm{id})$ and $\gamma_i = (\llbracket p_i \rrbracket, \llbracket p_i \rrbracket)$ for $i \in [1,n]$. To define $\beta_j$ and $\delta_j$ for $j \in [1,m]$ consider the clause $c_j$. There are $7$ truth assignments of the variables occuring in this clause that satisfy this clause. Let $\sigma_1,\dots,\sigma_7$ be the truth assignments of the variables occuring in this clause that satisfy the clause. Then we define for all $j \in [1,m]$ and $l \in [1,7]$ numbers $0 \leq z_{j,l} \leq q_j-1$ as the smallest positive integers satisfying the congruences
\begin{displaymath}
z_{j,l} \equiv \sigma_l(x_i) \bmod p_i
\end{displaymath}
for all $i \in I_j$. Then we define for $j \in [1,m]$
\begin{align*}
\beta_j&=(\llbracket  q_j \rrbracket^{z_{j,1}}, \llbracket  q_j \rrbracket^{z_{j,2}}, \llbracket  q_j \rrbracket^{z_{j,3}}, \llbracket  q_j \rrbracket^{z_{j,4}}, \llbracket  q_j \rrbracket^{z_{j,5}}, \llbracket  q_j \rrbracket^{z_{j,6}}, \llbracket  q_j \rrbracket^{z_{j,7}})\\
\delta_j&=(\llbracket  q_j \rrbracket, \llbracket  q_j \rrbracket, \llbracket  q_j \rrbracket, \llbracket  q_j \rrbracket, \llbracket  q_j \rrbracket, \llbracket  q_j \rrbracket, \llbracket  q_j \rrbracket).
\end{align*}
Finally we set $k = \sum_{i=1}^n p_i + 6\sum_{j=1}^m q_j$. Now we show that $C$ is satisfiable if and only if there is a number $z \in \mathbb{N}$ such that $H(\tau,\pi^z) \leq k$. Suppose $C$ is satisfiable and let $\sigma$ be a truth assignment that satisfies $C$. Let $0 \leq z \leq \prod_{i=1}^n p_i - 1$ be the smallest positive integer satisfying the congruence $z \equiv \sigma(x_i) \bmod p_i$ for $i \in [1,n]$. Consider $\alpha_i$ and $\gamma_i$. Clearly we have that $(\llbracket  p_i \rrbracket, \mathrm{id})$ and $(\llbracket  p_i \rrbracket, \llbracket  p_i \rrbracket)^z$ match at $p_i$ positions. Now consider $\beta_j$ and $\delta_j$ for some $j \in [1,m]$. Since $C$ is satisfied by $\sigma$ there is an $l \in [1,7]$ such that $\sigma_l(x_i) = \sigma(x_i)$ for $i \in I_j$. Hence we have $z \equiv z_{j,l} \bmod q_j$. Then we have $\llbracket  q_j \rrbracket^z = \llbracket  q_j \rrbracket^{z_{j,l}}$ which gives us that
\begin{displaymath}
\delta_j^z = (\llbracket  q_j \rrbracket^{z_{j,l}}, \llbracket  q_j \rrbracket^{z_{j,l}}, \llbracket  q_j \rrbracket^{z_{j,l}}, \llbracket  q_j \rrbracket^{z_{j,l}}, \llbracket  q_j \rrbracket^{z_{j,l}}, \llbracket  q_j \rrbracket^{z_{j,l}}, \llbracket  q_j \rrbracket^{z_{j,l}})
\end{displaymath}
and matches with
\begin{displaymath}
\beta_j = (\llbracket  q_j \rrbracket^{z_{j,1}}, \llbracket  q_j \rrbracket^{z_{j,2}}, \llbracket  q_j \rrbracket^{z_{j,3}}, \llbracket  q_j \rrbracket^{z_{j,4}}, \llbracket  q_j \rrbracket^{z_{j,5}}, \llbracket  q_j \rrbracket^{z_{j,6}}, \llbracket  q_j \rrbracket^{z_{j,7}})
\end{displaymath}
at $q_j$ positions. This gives us a total of $\sum_{i=1}^n p_i + \sum_{j=1}^m q_j$ matching positions. Subtracting this number from the total number of positions gives us
\begin{displaymath}
H(\tau,\pi^z) = 2\sum_{i=1}^n p_i + 7\sum_{j=1}^m q_j - (\sum_{i=1}^n p_i + \sum_{j=1}^m q_j) = \sum_{i=1}^n p_i + 6\sum_{j=1}^m q_j = k
\end{displaymath}
mismatches.

Vice versa suppose $H(\tau,\pi^z) \leq k$ for some $z \in \mathbb{N}$. Consider $\alpha_i$ and $\gamma_i$. By Lemma~\ref{lemmahammingdistance} we have that $(\llbracket  p_i \rrbracket, \mathrm{id})$ and $(\llbracket  p_i \rrbracket, \llbracket  p_i \rrbracket)^z$ match at $p_i$ positions if $z \equiv 0,1 \bmod p_i$ or at no position otherwise. Moreover
\begin{align*}
\delta_j^z&= (\llbracket  q_j \rrbracket^z, \llbracket  q_j \rrbracket^z, \llbracket  q_j \rrbracket^z, \llbracket  q_j \rrbracket^z, \llbracket  q_j \rrbracket^z, \llbracket  q_j \rrbracket^z, \llbracket  q_j \rrbracket^z)\\
\beta_j&= (\llbracket  q_j \rrbracket^{z_{j,1}}, \llbracket  q_j \rrbracket^{z_{j,2}}, \llbracket  q_j \rrbracket^{z_{j,3}}, \llbracket  q_j \rrbracket^{z_{j,4}}, \llbracket  q_j \rrbracket^{z_{j,5}}, \llbracket  q_j \rrbracket^{z_{j,6}}, \llbracket  q_j \rrbracket^{z_{j,7}})
\end{align*}
match at $q_j$ positions if $z \equiv z_{j,1},\dots,z_{j,7} \bmod q_j$ or at no position otherwise. By counting the number of possible matchings we find that we can match at most $\sum_{i=1}^n p_i + \sum_{j=1}^m q_j$ positions. By noting that $k+\sum_{i=1}^n p_i + \sum_{j=1}^m q_j$ equals the total number of positions we obtain that in every coordinate of $G$ we need the maximal number of matchings. Therefore we have for all $i \in [1,n]$ the congruence $z \equiv 0,1 \bmod p_i$. Therefore $z$ encodes a truth assignment of the variables. Since the $z_{j,l}$ encode satisfying truth assignments of $c_j$ we find that the truth assignment encoded by $z$ satisfies all clauses. Therefore we obtain by
\begin{displaymath}
\sigma(x_i) = \begin{cases}
1 & \text{if } z \equiv 1 \bmod p_i\\
0 & \text{if } z \equiv 0 \bmod p_i
\end{cases}
\end{displaymath}
a satisfying truth assignment for $C$.
\end{proof}

\subsection{Cayley Distance}
\begin{lemma}\label{lemmaprimelargeenough}
Let $n \geq 1$ be an integer. Let us denote by $\hat{p}_k$ the $k^{\text{th}}$ prime, i.e. $\hat{p}_1 = 2, \hat{p}_2 = 3, \dots$. Then $\hat{p}_{n^2+86}^3 > 6\hat{p}_{n^2+n+85}^2$.
\end{lemma}
\begin{proof}
We have
\begin{equation}\label{lowerboundprime}
\hat{p}_k \geq k(\ln k+\ln\ln k-1) \text{ for all } k \geq 2 \text{ by \cite[Theorem~3]{dusart}}.
\end{equation}
Moreover we have
\begin{displaymath}
\hat{p}_k \leq k(\ln k + \ln\ln k) \text{ if } 6 \leq k \leq e^{95} \text{ by \cite[Theorem~28]{rosser}}
\end{displaymath}
and
\begin{displaymath}
\hat{p}_k \leq k(\ln k + \ln\ln k - 0.9484) \text{ for all } k \geq 39017 \text{ by \cite[Chapter~4]{dusart}}
\end{displaymath}
which gives us
\begin{equation}\label{upperboundprime}
\hat{p}_k \leq k(\ln k + \ln\ln k) \text{ for all } k \geq 6.
\end{equation}
Using \eqref{lowerboundprime} we obtain
\begin{displaymath}
\hat{p}_{n^2+86}^3 \geq (n^2+86)^3(\ln(n^2+86)+\ln\ln(n^2+86)-1)^3
\end{displaymath}
and \eqref{upperboundprime} gives us
\begin{displaymath}
\hat{p}_{n^2+n+85}^2 \leq (n^2+n+85)^2(\ln(n^2+n+85) + \ln\ln(n^2+n+85))^2.
\end{displaymath}
From this it follows now that
\begin{align*}
\hat{p}_{n^2+86}^3&\geq (n^2+86)^3(\ln(n^2+86)+\ln\ln(n^2+86)-1)^3\\
&> (n^2+86)^3\ln(n^2+86)^3\\
&= (n^2+86)\ln(n^2+86)(n^2+86)^2\ln(n^2+86)^2\\
&> 384(n^2+86)^2\ln(n^2+86)^2\\
&= 6 \cdot 64(n^2+86)^2\ln(n^2+86)^2\\
&= 6 \cdot 4(n^2+86)^2 \cdot 16\ln(n^2+86)^2\\
&= 6(2(n^2+86))^2(2\ln(n^2+86)+2\ln(n^2+86))^2\\
&> 6(n^2+n+85)^2(\ln(n^2+n+85)+\ln\ln(n^2+n+85))^2\\
& \geq 6\hat{p}_{n^2+n+85}^2
\end{align*}
for all $n \geq 1$ which shows the lemma.
\end{proof}

\begin{remark}
Although the estimation $\hat{p}_{n^2+86}^3 > 6\hat{p}_{n^2+n+85}^2$ of Lemma~\ref{lemmaprimelargeenough} is not very accurate it is sufficient for our purposes. And in fact it can be shown that already $\hat{p}_{n+8}^3 > 6\hat{p}_{2n+7}^2$ for all $n \geq 1$ but a formal proof needs a more complicated technique.
\end{remark}

\begin{theorem}
The {\sc Subgroup distance problem} regarding the Cayley distance is {\sf NP}-complete when the input group is cyclic.
\end{theorem}
\begin{proof}
We give a log-space reduction from \textsf{X3HS}. Let $X$ be a finite set and $\mathcal{B} \subseteq 2^X$ be a set of subsets of $X$ all of size $3$. W.l.o.g. assume that $X = [1,n]$ and let $\mathcal{B} = \{C_1,\dots,C_m\}$. Let $p_1 < \dots < p_n$ be the first $n$ primes such that $p_1^3 > 6p_n^2$. Note that $p_1,p_n \in O(n^2 \log n)$ by Lemma~\ref{lemmaprimelargeenough} and the prime number theorem. We define $q_j = \prod_{i \in C_j} p_i$ for all $j \in [1,m]$. We will work with the group
\begin{displaymath}
G = \prod_{j=1}^m S_{q_j}^6
\end{displaymath}
which naturally embedds into $S_N$ for $N = 6\sum_{j=1}^m q_j$. Moreover for $j \in [1,m]$ and all $d \in [1,6]$ we define the number $0 \leq s_{j,d} < q_j$ as the smallest positive integer satisfying the congruences in which we assume $C_j = \{i_1,i_2,i_3\}$ with $i_1 < i_2 < i_3$
\begin{align*}
s_{j,1}&\equiv 1 \bmod p_{i_1} && s_{j,2} \equiv 0 \bmod p_{i_1} && s_{j,3} \equiv 0 \bmod p_{i_1}\\
s_{j,1}&\equiv 0 \bmod p_{i_2} && s_{j,2} \equiv 1 \bmod p_{i_2} && s_{j,3} \equiv 0 \bmod p_{i_2}\\
s_{j,1}&\equiv 0 \bmod p_{i_3} && s_{j,2} \equiv 0 \bmod p_{i_3} && s_{j,3} \equiv 1 \bmod p_{i_3}\\
\\[3mm]
s_{j,4}&\equiv 1 \bmod p_{i_1} && s_{j,5} \equiv 3 \bmod p_{i_1} && s_{j,6} \equiv 2 \bmod p_{i_1}\\
s_{j,4}&\equiv 2 \bmod p_{i_2} && s_{j,5} \equiv 1 \bmod p_{i_2} && s_{j,6} \equiv 3 \bmod p_{i_2}\\
s_{j,4}&\equiv 3 \bmod p_{i_3} && s_{j,5} \equiv 2 \bmod p_{i_3} && s_{j,6} \equiv 1 \bmod p_{i_3}.
\end{align*}
We define the input group elements $\tau, \pi \in G$ as follows where $j$ ranges over $[1,m]$:
\begin{align*}
\tau &= (\tau_1,\dots,\tau_m)\\
\tau_j &= (\llbracket q_j \rrbracket^{s_{j,1}}, \llbracket q_j \rrbracket^{s_{j,2}}, \llbracket q_j \rrbracket^{s_{j,3}}, \llbracket q_j \rrbracket^{s_{j,4}}, \llbracket q_j \rrbracket^{s_{j,5}}, \llbracket q_j \rrbracket^{s_{j,6}})\\
&\\
\pi &= (\pi_1,\dots,\pi_m)\\
\pi_j &= (\llbracket q_j \rrbracket, \llbracket q_j \rrbracket, \llbracket q_j \rrbracket, \llbracket q_j \rrbracket, \llbracket q_j \rrbracket, \llbracket q_j \rrbracket)
\end{align*}
and we define
\begin{displaymath}
k = N - \sum_{j=1}^m (q_j + 2 + \sum_{i \in C_j} p_i).
\end{displaymath}
Now we will show there is $x \in \mathbb{N}$ such that $C(\tau,\pi^x) \leq k$ if and only if there is a subset $X' \subseteq X$ such that $|X' \cap C_j| = 1$ for all $j \in [1,m]$.

Suppose there is $x \in \mathbb{N}$ such that $C(\tau,\pi^x) \leq k$. We define
\begin{displaymath}
X' = \{i \in [1,n] \mid x \equiv 1 \bmod p_i\}.
\end{displaymath}
\begin{claim}\label{claimsplitupperboundcycles}
For all $j \in [1,m]$ and all $z \in \mathbb{Z}$ we have that $\tau_j\pi_j^{-z}$ splits into exactly $q_j + 2 + \sum_{i \in C_j} p_i$ cycles if there is $a \in [1,3]$ such that $z \equiv s_{j,a} \bmod q_j$ or in strictly less than $q_j + 2 + \sum_{i \in C_j} p_i$ cycles if $z \not \equiv s_{j,a} \bmod q_j$ for all $a \in [1,3]$.
\end{claim}
Let $j \in [1,m]$ and assume $C_j = \{i_1,i_2,i_3\}$ with $i_1 < i_2 < i_3$. Note that for all $d \in [1,6]$ we have that $\llbracket q_j \rrbracket^{s_{j,d}-z}$ will split into $\gcd(q_j,s_{j,d}-z)$ cycles of length $\frac{q_j}{\gcd(q_j,s_{j,d}-z)}$ by Lemma~\ref{lemmasplit}.

Suppose there is an $a \in [1,3]$ such that $z \equiv s_{j,a} \bmod q_j$. Then clearly $z \not \equiv s_{j,c} \bmod q_j$ for all $c \in [1,6] \setminus \{a\}$ since $s_{j,e} \not \equiv s_{j,f} \bmod q_j$ for all $e \neq f$. Moreover we have for all $b \in [1,3] \setminus \{a\}$ and all $c \in [1,3]$
\begin{displaymath}
s_{j,b+3} - z \equiv s_{j,b+3} - s_{j,a} \not \equiv 0 \bmod p_{i_c}
\end{displaymath}
and hence $\llbracket q_j \rrbracket^{s_{j,b+3}-z}$ will not split into further cycles by Lemma~\ref{lemmasplit}. Moreover we have for all $b \in [1,3] \setminus \{a\}$
\begin{displaymath}
s_{j,b} - z \equiv s_{j,b} - s_{j,a} \begin{cases}
\equiv 0 \bmod p_{i_c} & \text{if } c \in [1,3] \setminus \{a,b\}\\
\not \equiv 0 \bmod p_{i_c} & \text{if } c \in \{a,b\}
\end{cases}
\end{displaymath}
and hence $\llbracket q_j \rrbracket^{s_{j,b}-z}$ will split into $p_{i_c}$ cycles by Lemma~\ref{lemmasplit} with $c \in [1,3] \setminus \{a,b\}$. Moreover we have
\begin{displaymath}
s_{j,a} - z \equiv s_{j,a} - s_{j,a} \equiv 0 \bmod q_j
\end{displaymath}
and hence $\llbracket q_j \rrbracket^{s_{j,a}-z}$ will split into $q_j$ fixed points by Lemma~\ref{lemmasplit}. Finally we have
\begin{displaymath}
s_{j,a+3} - z \equiv s_{j,a+3} - s_{j,a} \equiv 1 - 1 \equiv 0 \bmod p_{i_a}
\end{displaymath}
and for all $b \in [1,3] \setminus \{a\}$ we have
\begin{displaymath}
s_{j,a+3} - z \equiv s_{j,a+3} - s_{j,a} \equiv s_{j,a+3} - 0 \not \equiv 0 \bmod p_{i_b}
\end{displaymath}
and hence $\llbracket q_j \rrbracket^{s_{j,a+3}-z}$ will split into $p_{i_a}$ cycles by Lemma~\ref{lemmasplit}. Thus the total number of cycles in $\tau_j\pi_j^{-z}$ is
\begin{displaymath}
q_j + 2 + \sum_{i \in C_j} p_i.
\end{displaymath}

Suppose $z \not \equiv s_{j,a} \bmod q_j$ for all $a \in [1,3]$. If also $z \not \equiv s_{j,a} \bmod q_j$ for all $a \in [4,6]$ then $\tau_j\pi_j^{-z}$ can only split into at most $6p_n^2$ cycles which is strictly less than $q_j + 2 + \sum_{i \in C_j} p_i$ since we already have
\begin{displaymath}
6p_n^2 < p_1^3 < q_j.
\end{displaymath}
In the case $z \equiv s_{j,a} \bmod q_j$ for some $a \in [4,6]$ we have $z \not \equiv s_{j,c} \bmod q_j$ for all $c \in [1,6] \setminus \{a\}$ since $s_{j,e} \not \equiv s_{j,f} \bmod q_j$ for all $e \neq f$. Moreover we have for all $b \in [1,3] \setminus \{a-3\}$ and all $c \in [1,3]$
\begin{displaymath}
s_{j,b} - z \equiv s_{j,b} - s_{j,a} \not \equiv 0 \bmod p_{i_c}
\end{displaymath}
and hence $\llbracket q_j \rrbracket^{s_{j,b}-z}$ will not split into further cycles by Lemma~\ref{lemmasplit}. Similarly we have for all $b \in [4,6] \setminus \{a\}$ and all $c \in [1,3]$
\begin{displaymath}
s_{j,b} - z \equiv s_{j,b} - s_{j,a} \not \equiv 0 \bmod p_{i_c}
\end{displaymath}
and hence also in this case $\llbracket q_j \rrbracket^{s_{j,b}-z}$ will not split into further cycles by Lemma~\ref{lemmasplit}. Moreover we have
\begin{displaymath}
s_{j,a} - z \equiv s_{j,a} - s_{j,a} \equiv 0 \bmod q_j
\end{displaymath}
and hence $\llbracket q_j \rrbracket^{s_{j,a}-z}$ will split into $q_j$ fixed points by Lemma~\ref{lemmasplit}. Finally we have
\begin{displaymath}
s_{j,a-3} - z \equiv s_{j,a-3} - s_{j,a} \equiv 1 - 1 \equiv 0 \bmod p_{i_{a-3}}
\end{displaymath}
and for all $b \in [1,3] \setminus \{a-3\}$ we have
\begin{displaymath}
s_{j,a-3} - z \equiv s_{j,a-3} - s_{j,a} \equiv 0 - s_{j,a} \not \equiv 0 \bmod p_{i_b}
\end{displaymath}
and hence $\llbracket q_j \rrbracket^{s_{j,a-3}-z}$ will split into $p_{i_{a-3}}$ cycles by Lemma~\ref{lemmasplit}. This gives us a total of
\begin{displaymath}
4 + q_j + p_{i_{a-3}} < q_j + 2 + \sum_{i \in C_j} p_i
\end{displaymath}
cycles.
\qed

\begin{claim}\label{claimexactlyonea}
For all $j \in [1,m]$ there is exactly one $a \in [1,3]$ such that $x \equiv 1 \bmod p_{i_a}$ and $x \equiv 0 \bmod p_{i_b}$ for all $b \in [1,3] \setminus \{a\}$ in which $C_j = \{i_1,i_2,i_3\}$ with $i_1 < i_2 < i_3$.
\end{claim}
By Claim~\ref{claimsplitupperboundcycles} we find that summing up the largest possible amount of splitting cycles gives us
\begin{displaymath}
C(\tau,\pi^x) \geq N - \sum_{j=1}^m (q_j + 2 + \sum_{i \in C_j} p_i) = k
\end{displaymath}
and hence $C(\tau,\pi^x) = k$. Thus for all $j \in [1,m]$ the only possibility for $x$ is to satisfy $x \equiv s_{j,a} \bmod q_j$ for exactly one $a \in [1,3]$ which implies $x \equiv 1 \bmod p_{i_a}$ and $x \equiv 0 \bmod p_{i_b}$ for all $b \in [1,3] \setminus \{a\}$ as claimed.
\qed

Now we will show $|X' \cap C_j| = 1$ for all $j \in [1,m]$. Let $C_j = \{i_1,i_2,i_3\}$ with $i_1 < i_2 < i_3$. Then by Claim~\ref{claimexactlyonea} there is exactly one $a \in [1,3]$ such that $x \equiv 1 \bmod p_{i_a}$ and $x \equiv 0 \bmod p_{i_b}$ for all $b \in [1,3] \setminus \{a\}$. Thus we have $i_a \in X'$ and $i_b \notin X'$ for all $b \in [1,3] \setminus \{a\}$ which finally gives us $|X' \cap C_j| = 1$.

Vice versa suppose there is a subset $X' \subseteq X$ such that $|X' \cap C_j| = 1$ for all $j \in [1,m]$. Then we define $x$ as the smallest positive integer satisfying
\begin{displaymath}
x \equiv \begin{cases}
1 \bmod p_i & \text{if } i \in X'\\
0 \bmod p_i & \text{if } i \notin X'
\end{cases}
\end{displaymath}
for all $i \in [1,n]$. Then we obtain for all $j \in [1,m]$ and $i \in [1,n]$
\begin{displaymath}
x \equiv \begin{cases}
1 \bmod p_i & \text{if } i \in X' \cap C_j\\
0 \bmod p_i & \text{if } i \in C_j \setminus X'
\end{cases}
\end{displaymath}
from which it follows that $x \equiv s_{j,a} \bmod q_j$ where $a$ is the unique element in $X' \cap C_j$. Then $\tau_j\pi_j^{-x}$ splits into exactly $q_j + 2 + \sum_{i \in C_j} p_i$ cycles by Claim~\ref{claimsplitupperboundcycles} for all $j \in [1,m]$ which gives us
\begin{displaymath}
C(\tau,\pi^x) = N - \sum_{j=1}^m (q_j + 2 + \sum_{i \in C_j} p_i) = k.
\end{displaymath}
This shows the theorem.
\end{proof}

\subsection{$l_\infty$ Distance}
\subsubsection{General Case}
\begin{lemma}\label{pkinftyleqk}
Let $p \geq 5$ be an odd prime and $k \geq 2$ be a non-negative integer. Define
\begin{displaymath}
\delta = (1, k+1, 2k+1, \dots, \frac{p-1}{2}k+1, \frac{p-1}{2}k, \frac{p-3}{2}k, \frac{p-5}{2}k, \dots, k) \in S_{\frac{p-1}{2}k+1}
\end{displaymath}
in which $\delta$ is a cycle of length $p$. Then $l_\infty((\delta, \mathrm{id}), (\delta, \delta)^x) \leq k$ if and only if $x \equiv 0,1 \bmod p$.
\end{lemma}
\begin{proof}
One direction is clear since the difference of two consecutive numbers of $\delta$ is at most $k$. Now suppose $l_\infty((\delta, \mathrm{id}), (\delta, \delta)^x) \leq k$. It suffices to show for all $a \in [2,p-1]$ if $x \equiv a \bmod p$ then $l_\infty((\delta, \mathrm{id}), (\delta, \delta)^x) > k$. In the case $2 \leq a \leq \frac{p-1}{2}$ we have $(1,1)^{(\delta, \mathrm{id})} = (k+1,1)$ and $(1,1)^{(\delta, \delta)^a} = (ak+1,ak+1)$. Therefore the distance is at least $ak+1-1 = ak \geq 2k$. In the case $\frac{p+1}{2} \leq a \leq p-2$ we have $(1,1)^{(\delta, \delta)^a} = (k(p-a),k(p-a))$. In this case the distance is at least $k(p-a) - 1 \geq k(p-(p-2)) - 1 = 2k-1$. In the case $a=p-1$ we have $(k+1,k+1)^{(\delta, \mathrm{id})} = (2k+1,k+1)$ and $(k+1,k+1)^{(\delta, \delta)^{p-1}} = (1,1)$ which gives us a distance of $2k+1-1 = 2k$.
\end{proof}

\begin{theorem}
The {\sc Subgroup distance problem} regarding the $l_\infty$ distance is ${\sf NP}$-complete when the input group is cyclic.
\end{theorem}
\begin{proof}
We give a log-space reduction from \textsc{3-SAT}. Let $X = \{x_1,\dots,x_n\}$ be a set of variables and let $C = \{c_1,\dots,c_m\}$ be a set of clauses over $X$ where $c_j$ contains exactly $3$ different literals for all $j \in [1,m]$. W.l.o.g. we can assume that no clause contains a positive and a negative literal regarding the same variable. For $j=1,\dots,m$ we define $I_j \subseteq [1,n]$ as the set of all indices $i$ such that $c_j$ contains $x_i$ or $\bar{x}_i$. Let $p_1,\dots,p_n$ be the first $n$ odd primes with $p_1 \geq 5$. Moreover let $k = p_n^3$, $q_j = \prod_{i \in I_j} p_i$ for $j=1,\dots,m$ and let $N = \sum_{i=1}^n ((p_i-1)k+2) + m(k+2)$. We will work with the group $G \leq S_N$ in which
\begin{displaymath}
G = \prod_{i=1}^n V_i \times \prod_{j=1}^m U_j
\end{displaymath}
with $V_i = S_{\frac{p_i-1}{2}k+1}^2$ and $U_j = S_{k+2}$. For $i=1,\dots,n$ we define the cycle $\delta_i$ of length $p_i$ by
\begin{displaymath}
\delta_i = (1, k+1, 2k+1, \dots, \frac{p_i-1}{2}k+1, \frac{p_i-1}{2}k, \frac{p_i-3}{2}k, \frac{p_i-5}{2}k, \dots, k).
\end{displaymath}
Now we define the input group elements as
\begin{align*}
\tau&= (\zeta_1,\dots,\zeta_n,\mu_1,\dots,\mu_m)\\
\pi&= (\eta_1,\dots,\eta_n,\lambda_1,\dots,\lambda_m)
\end{align*}
with $\zeta_i = (\delta_i, \mathrm{id})$ and $\eta_i = (\delta_i, \delta_i)$ for $i \in [1,n]$. To define $\lambda_j$ and $\mu_j$ for $j \in [1,m]$ we first define some auxiliary permutations. Let $j \in [1,m]$ and let $d < e < f \in I_j$ be the indices of the variables that occur (negated or unnegated) in this clause. Then we define permutations that do not need to be constructed explicitly:
\begin{align*}
\alpha_j&= \prod_{r=1}^{p_f} \alpha_{j,r} && \beta_j = \prod_{r=1}^{p_f} \beta_{j,r} && \gamma_j = \prod_{s=1}^{p_e} \gamma_{j,s}\\
\alpha_{j,r}&= \prod_{s=1}^{p_e} \alpha_{j,r,s} && \beta_{j,r} = \prod_{t=1}^{p_d} \beta_{j,r,t} && \gamma_{j,s} = \prod_{t=1}^{p_d} \gamma_{j,s,t}\\
\alpha_{j,r,s}&= (\alpha_{j,r,s,1},\dots,\alpha_{j,r,s,p_d}) && \beta_{j,r,t} = (\beta_{j,r,t,1},\dots,\beta_{j,r,t,p_e}) && \gamma_{j,s,t} = (\gamma_{j,s,t,1},\dots,\gamma_{j,s,t,p_f})
\end{align*}
with $\alpha_{j,r,s,t} \in [1,q_j]$ and $\alpha_{j,r,s,t} \neq \alpha_{j,r',s',t'}$ for $(r,s,t) \neq (r',s',t')$ and the constraint
\begin{equation}\label{alphabetagammaconstraint}
\alpha_{j,r,s,t} = \beta_{j,r,t,s} = \gamma_{j,s,t,r}
\end{equation}
for $r \in [1,p_f], s \in [1,p_e]$ and $t \in [1,p_d]$. Note that $\ord(\alpha_j)=p_d,\ord(\beta_j)=p_e$ and $\ord(\gamma_j)=p_f$. We fix the following $8$ values:
\begin{equation}\label{fixedvalues}
\begin{split}
\alpha_{j,1,1,2}&= 1\\
\alpha_{j,1,1,1}&= 2\\
\alpha_{j,1,p_e,2}&= 3\\
\alpha_{j,p_f,1,2}&= 4\\
\alpha_{j,p_f,p_e,2}&= 5\\
\alpha_{j,p_f,1,1}&= 6\\
\alpha_{j,1,p_e,1}&= 7\\
\alpha_{j,p_f,p_e,1}&= 8.
\end{split}
\end{equation}
In the clause $c_j$ there is exactly one truth assignment of the variables occuring in $c_j$ that does not satisfy this clause. Let $\sigma_j$ denote this partial truth assignment. We define
\begin{displaymath}
w_j = \begin{cases}
1 & \text{~~~~if } \sigma_j(x_d)=0, \sigma_j(x_e)=0 \text{ and } \sigma_j(x_f)=0\\
2 & \text{~~~~if } \sigma_j(x_d)=1, \sigma_j(x_e)=0 \text{ and } \sigma_j(x_f)=0\\
3 & \text{~~~~if } \sigma_j(x_d)=0, \sigma_j(x_e)=1 \text{ and } \sigma_j(x_f)=0\\
4 & \text{~~~~if } \sigma_j(x_d)=0, \sigma_j(x_e)=0 \text{ and } \sigma_j(x_f)=1\\
5 & \text{~~~~if } \sigma_j(x_d)=0, \sigma_j(x_e)=1 \text{ and } \sigma_j(x_f)=1\\
6 & \text{~~~~if } \sigma_j(x_d)=1, \sigma_j(x_e)=0 \text{ and } \sigma_j(x_f)=1\\
7 & \text{~~~~if } \sigma_j(x_d)=1, \sigma_j(x_e)=1 \text{ and } \sigma_j(x_f)=0\\
8 & \text{~~~~if } \sigma_j(x_d)=1, \sigma_j(x_e)=1 \text{ and } \sigma_j(x_f)=1
\end{cases}
\end{displaymath}
and finally we define $\lambda_j = \alpha_j\beta_j\gamma_j(k+2,k)$ and $\mu_j = (w_j,k+2)$. Now we show that we can construct $\alpha_j\beta_j\gamma_j$ in log-space.
\begin{claim}\label{claimjalphabetagammecommute}
$\alpha_j, \beta_j, \gamma_j$ pairwise commute.
\end{claim}
In the following we make use of Constraint~\eqref{alphabetagammaconstraint} several times without explicit mentioning. We have
\begin{align*}
\alpha_{j,r,s,t}^{\beta_j\alpha_j}=\beta_{j,r,t,s}^{\beta_j\alpha_j}&= \begin{cases}
\mathrlap{\beta_{j,r,t,1}^{\alpha_j}}\hphantom{\alpha_{j,r,s+1,t+1}} = \mathrlap{\alpha_{j,r,1,t}^{\alpha_j}}\hphantom{\beta_{j,r,t+1,s+1}} & \text{if } s=p_e\\
\mathrlap{\beta_{j,r,t,s+1}^{\alpha_j}}\hphantom{\alpha_{j,r,s+1,t+1}} = \mathrlap{\alpha_{j,r,s+1,t}^{\alpha_j}}\hphantom{\beta_{j,r,t+1,s+1}} & \text{if } 1 \leq s < p_e
\end{cases}\\
&= \begin{cases}
\mathrlap{\alpha_{j,r,1,1}}\hphantom{\alpha_{j,r,s+1,t+1}} = \beta_{j,r,1,1} & \text{if } t=p_d, s=p_e\\
\mathrlap{\alpha_{j,r,1,t+1}}\hphantom{\alpha_{j,r,s+1,t+1}} = \beta_{j,r,t+1,1} & \text{if } 1 \leq t < p_d, s=p_e\\
\mathrlap{\alpha_{j,r,s+1,1}}\hphantom{\alpha_{j,r,s+1,t+1}} = \beta_{j,r,1,s+1} & \text{if } t=p_d, 1 \leq s < p_e\\
\alpha_{j,r,s+1,t+1} = \beta_{j,r,t+1,s+1} & \text{if } 1 \leq t < p_d, 1 \leq s < p_e
\end{cases}\\
&=\begin{cases}
\mathrlap{\beta_{j,r,1,s}^{\beta_j}}\hphantom{\alpha_{j,r,s+1,t+1}} = \mathrlap{\alpha_{j,r,s,1}^{\beta_j}}\hphantom{\beta_{j,r,t+1,s+1}} & \text{if } t=p_d\\
\mathrlap{\beta_{j,r,t+1,s}^{\beta_j}}\hphantom{\alpha_{j,r,s+1,t+1}} = \mathrlap{\alpha_{j,r,s,t+1}^{\beta_j}}\hphantom{\beta_{j,r,t+1,s+1}} & \text{if } 1 \leq t < p_d
\end{cases}\\
&=\alpha_{j,r,s,t}^{\alpha_j\beta_j}.
\end{align*}
Analogously we obtain that $\alpha_j,\gamma_j$ and $\beta_j,\gamma_j$ commute.
\qed

By Claim~\ref{claimjalphabetagammecommute} we have $\ord(\alpha_j\beta_j\gamma_j) = p_dp_ep_f = q_j$ from which it follows that $\alpha_j\beta_j\gamma_j$ is a cycle of length $q_j$. Now we give a mapping to construct $\alpha_j\beta_j\gamma_j$ in log-space:
\begin{align*}
\alpha_{j,r,s,t}^{\alpha_j\beta_j\gamma_j}&= \begin{cases}
\mathrlap{\alpha_{j,r,s,1}^{\beta_j\gamma_j}}\hphantom{\gamma_{j,s+1,t+1,r+1}} = \mathrlap{\beta_{j,r,1,s}^{\beta_j\gamma_j}}\hphantom{\alpha_{j,r+1,s+1,t+1}} & \text{if } t=p_d\\
\mathrlap{\alpha_{j,r,s,t+1}^{\beta_j\gamma_j}}\hphantom{\gamma_{j,s+1,t+1,r+1}} = \mathrlap{\beta_{j,r,t+1,s}^{\beta_j\gamma_j}}\hphantom{\alpha_{j,r+1,s+1,t+1}} & \text{if } 1 \leq t < p_d
\end{cases}\\
&= \begin{cases}
\mathrlap{\beta_{j,r,1,1}^{\gamma_j}}\hphantom{\gamma_{j,s+1,t+1,r+1}} = \mathrlap{\gamma_{j,1,1,r}^{\gamma_j}}\hphantom{\alpha_{j,r+1,s+1,t+1}} & \text{if } t=p_d, s=p_e\\
\mathrlap{\beta_{j,r,1,s+1}^{\gamma_j}}\hphantom{\gamma_{j,s+1,t+1,r+1}} = \mathrlap{\gamma_{j,s+1,1,r}^{\gamma_j}}\hphantom{\alpha_{j,r+1,s+1,t+1}} & \text{if } t=p_d, 1 \leq s < p_e\\
\mathrlap{\beta_{j,r,t+1,1}^{\gamma_j}}\hphantom{\gamma_{j,s+1,t+1,r+1}} = \mathrlap{\gamma_{j,1,t+1,r}^{\gamma_j}}\hphantom{\alpha_{j,r+1,s+1,t+1}} & \text{if } 1 \leq t < p_d, s=p_e\\
\mathrlap{\beta_{j,r,t+1,s+1}^{\gamma_j}}\hphantom{\gamma_{j,s+1,t+1,r+1}} = \mathrlap{\gamma_{j,s+1,t+1,r}^{\gamma_j}}\hphantom{\alpha_{j,r+1,s+1,t+1}} & \text{if } 1 \leq t < p_d, 1 \leq s < p_e
\end{cases}\\
&= \begin{cases}
\mathrlap{\gamma_{j,1,1,1}}\hphantom{\gamma_{j,s+1,t+1,r+1}} = \alpha_{j,1,1,1} & \text{if } t=p_d, s=p_e, r=p_f\\
\mathrlap{\gamma_{j,1,1,r+1}}\hphantom{\gamma_{j,s+1,t+1,r+1}} = \alpha_{j,r+1,1,1} & \text{if } t=p_d, s=p_e, 1 \leq r < p_f\\
\mathrlap{\gamma_{j,s+1,1,1}}\hphantom{\gamma_{j,s+1,t+1,r+1}} = \alpha_{j,1,s+1,1} & \text{if } t=p_d, 1 \leq s < p_e, r=p_f\\
\mathrlap{\gamma_{j,s+1,1,r+1}}\hphantom{\gamma_{j,s+1,t+1,r+1}} = \alpha_{j,r+1,s+1,1} & \text{if } t=p_d, 1 \leq s < p_e, 1 \leq r < p_f\\
\mathrlap{\gamma_{j,1,t+1,1}}\hphantom{\gamma_{j,s+1,t+1,r+1}} = \alpha_{j,1,1,t+1} & \text{if } 1 \leq t < p_d, s=p_e, r=p_f\\
\mathrlap{\gamma_{j,1,t+1,r+1}}\hphantom{\gamma_{j,s+1,t+1,r+1}} = \alpha_{j,r+1,1,t+1} & \text{if } 1 \leq t < p_d, s=p_e, 1 \leq r < p_f\\
\mathrlap{\gamma_{j,s+1,t+1,1}}\hphantom{\gamma_{j,s+1,t+1,r+1}} = \alpha_{j,1,s+1,t+1} & \text{if } 1 \leq t < p_d, 1 \leq s < p_e, r=p_f\\
\gamma_{j,s+1,t+1,r+1} = \alpha_{j,r+1,s+1,t+1} & \text{if } 1 \leq t < p_d, 1 \leq s < p_e, 1 \leq r < p_f.
\end{cases}
\end{align*}
Because $\alpha_j\beta_j\gamma_j$ is a cycle we can start with an arbitrary triple $(r,s,t)$ and write in the output the numbers from $9$ up to $q_j$. When we obtain a triple where we already assigned a fixed value (see \eqref{fixedvalues}) we write in the output that fixed value instead. By this procedure we clearly can write $\alpha_j\beta_j\gamma_j$ in the output in log-space. Moreover $\alpha_j\beta_j\gamma_j$ evaluates as follows
\begin{equation}\label{equationevaluationjalphabetagamma}
\begin{split}
1^{\alpha_j^0\beta_j^0\gamma_j^0}&= 1\\
2^{\alpha_j^1\beta_j^0\gamma_j^0}&= 1\\
3^{\alpha_j^0\beta_j^1\gamma_j^0}&= 1\\
4^{\alpha_j^0\beta_j^0\gamma_j^1}&= 1\\
5^{\alpha_j^0\beta_j^1\gamma_j^1}&= 1\\
6^{\alpha_j^1\beta_j^0\gamma_j^1}&= 1\\
7^{\alpha_j^1\beta_j^1\gamma_j^0}&= 1\\
8^{\alpha_j^1\beta_j^1\gamma_j^1}&= 1
\end{split}
\end{equation}
since
\begin{align*}
1^{\alpha_j^0\beta_j^0\gamma_j^0}&=\mathrlap{1^{\mathrm{id}}}\hphantom{8^{\alpha_j\beta_j\gamma_j}}= 1\\
2^{\alpha_j^1\beta_j^0\gamma_j^0}&=\mathrlap{2^{\alpha_j}}\hphantom{8^{\alpha_j\beta_j\gamma_j}} = \mathrlap{\alpha_{j,1,1,1}^{\alpha_j}}\hphantom{\alpha_{j,p_f,p_e,1}^{\alpha_j\beta_j\gamma_j}} = \mathrlap{\alpha_{j,1,1,2}}\hphantom{\alpha_{j,p_f,p_e,2}^{\beta_j\gamma_j}} = 1\\
3^{\alpha_j^0\beta_j^1\gamma_j^0}&=\mathrlap{3^{\beta_j}}\hphantom{8^{\alpha_j\beta_j\gamma_j}} = \mathrlap{\alpha_{j,1,p_e,2}^{\beta_j}}\hphantom{\alpha_{j,p_f,p_e,1}^{\alpha_j\beta_j\gamma_j}} = \mathrlap{\beta_{j,1,2,p_e}^{\beta_j}}\hphantom{\alpha_{j,p_f,p_e,2}^{\beta_j\gamma_j}} = \mathrlap{\beta_{j,1,2,1}}\hphantom{\beta_{j,p_f,2,p_e}^{\beta_j\gamma_j}} = \mathrlap{\alpha_{j,1,1,2}}\hphantom{\beta_{j,p_f,2,1}^{\gamma_j}} = 1\\
4^{\alpha_j^0\beta_j^0\gamma_j^1}&=\mathrlap{4^{\gamma_j}}\hphantom{8^{\alpha_j\beta_j\gamma_j}} = \mathrlap{\alpha_{j,p_f,1,2}^{\gamma_j}}\hphantom{\alpha_{j,p_f,p_e,1}^{\alpha_j\beta_j\gamma_j}} = \mathrlap{\gamma_{j,1,2,p_f}^{\gamma_j}}\hphantom{\alpha_{j,p_f,p_e,2}^{\beta_j\gamma_j}} = \mathrlap{\gamma_{j,1,2,1}}\hphantom{\beta_{j,p_f,2,p_e}^{\beta_j\gamma_j}} = \mathrlap{\alpha_{j,1,1,2}}\hphantom{\beta_{j,p_f,2,1}^{\gamma_j}} = 1\\
5^{\alpha_j^0\beta_j^1\gamma_j^1}&=\mathrlap{5^{\beta_j\gamma_j}}\hphantom{8^{\alpha_j\beta_j\gamma_j}} = \mathrlap{\alpha_{j,p_f,p_e,2}^{\beta_j\gamma_j}}\hphantom{\alpha_{j,p_f,p_e,1}^{\alpha_j\beta_j\gamma_j}} = \mathrlap{\beta_{j,p_f,2,p_e}^{\beta_j\gamma_j}}\hphantom{\alpha_{j,p_f,p_e,2}^{\beta_j\gamma_j}} = \mathrlap{\beta_{j,p_f,2,1}^{\gamma_j}}\hphantom{\beta_{j,p_f,2,p_e}^{\beta_j\gamma_j}} = \mathrlap{\gamma_{j,1,2,p_f}^{\gamma_j}}\hphantom{\beta_{j,p_f,2,1}^{\gamma_j}} = \mathrlap{\gamma_{j,1,2,1}}\hphantom{\gamma_{j,1,2,p_f}^{\gamma_j}} = \alpha_{j,1,1,2} = 1\\
6^{\alpha_j^1\beta_j^0\gamma_j^1}&=\mathrlap{6^{\alpha_j\gamma_j}}\hphantom{8^{\alpha_j\beta_j\gamma_j}} = \mathrlap{\alpha_{j,p_f,1,1}^{\alpha_j\gamma_j}}\hphantom{\alpha_{j,p_f,p_e,1}^{\alpha_j\beta_j\gamma_j}} = \mathrlap{\alpha_{j,p_f,1,2}^{\gamma_j}}\hphantom{\alpha_{j,p_f,p_e,2}^{\beta_j\gamma_j}} = \mathrlap{\gamma_{j,1,2,p_f}^{\gamma_j}}\hphantom{\beta_{j,p_f,2,p_e}^{\beta_j\gamma_j}} = \mathrlap{\gamma_{j,1,2,1}}\hphantom{\beta_{j,p_f,2,1}^{\gamma_j}} = \mathrlap{\alpha_{j,1,1,2}}\hphantom{\gamma_{j,1,2,p_f}^{\gamma_j}} = 1\\
7^{\alpha_j^1\beta_j^1\gamma_j^0}& =\mathrlap{7^{\alpha_j\beta_j}}\hphantom{8^{\alpha_j\beta_j\gamma_j}} = \mathrlap{\alpha_{j,1,p_e,1}^{\alpha_j\beta_j}}\hphantom{\alpha_{j,p_f,p_e,1}^{\alpha_j\beta_j\gamma_j}} = \mathrlap{\alpha_{j,1,p_e,2}^{\beta_j}}\hphantom{\alpha_{j,p_f,p_e,2}^{\beta_j\gamma_j}} = \mathrlap{\beta_{j,1,2,p_e}^{\beta_j}}\hphantom{\beta_{j,p_f,2,p_e}^{\beta_j\gamma_j}} = \mathrlap{\beta_{j,1,2,1}}\hphantom{\beta_{j,p_f,2,1}^{\gamma_j}} = \mathrlap{\alpha_{j,1,1,2}}\hphantom{\gamma_{j,1,2,p_f}^{\gamma_j}} = 1\\
8^{\alpha_j^1\beta_j^1\gamma_j^1}&= 8^{\alpha_j\beta_j\gamma_j} = \alpha_{j,p_f,p_e,1}^{\alpha_j\beta_j\gamma_j} = \alpha_{j,p_f,p_e,2}^{\beta_j\gamma_j} = \beta_{j,p_f,2,p_e}^{\beta_j\gamma_j} = \beta_{j,p_f,2,1}^{\gamma_j} = \gamma_{j,1,2,p_f}^{\gamma_j} = \gamma_{j,1,2,1} = \alpha_{j,1,1,2} = 1.
\end{align*}
Now we will show there is a $z \in \mathbb{N}$ such that $l_\infty(\tau,\pi^z) \leq k$ if and only if $C$ is satisfiable. Suppose there is such a $z$. Consider the computations in $V_i$. By Lemma~\ref{pkinftyleqk} we have $l_\infty(\zeta_i,\eta_i^z) \leq k$ if and only if $z \equiv 0,1 \bmod p_i$. Now we consider the computations in $U_j$. We have $\lambda_j^z = (\alpha_j\beta_j\gamma_j)^z(k+2,k)^z$ and $\mu_j = (w_j,k+2)$. By Claim~\ref{claimjalphabetagammecommute} we have that $\alpha_j,\beta_j,\gamma_j$ pairwise commute which gives us $(\alpha_j\beta_j\gamma_j)^z = \alpha_j^z\beta_j^z\gamma_j^z$. Now let $z_1,z_2,z_3 \in \{0,1\}$ be such that $z_1 \equiv z \bmod p_d, z_2 \equiv z \bmod p_e$ and $z_3 \equiv z \bmod p_f$ in which $d < e < f \in I_j$. Such numbers exist since we have $z \equiv 0,1 \bmod p_i$ for all $i \in [1,n]$. Then we have $\alpha_j^z\beta_j^z\gamma_j^z = \alpha_j^{z_1}\beta_j^{z_2}\gamma_j^{z_3}$. By \eqref{equationevaluationjalphabetagamma} there is a $w \in [1,8]$ such that $w^{\alpha_j^{z_1}\beta_j^{z_2}\gamma_j^{z_3}} = 1$. If $w=w_j$ we get by $w^{\mu_j} = w_j^{\mu_j} = k+2$ a distance of $k+1$ contradicting $l_\infty(\tau,\pi^z) \leq k$. Therefore we have $w \neq w_j$. Since however $w_j$ is associated with a truth assignment that does not satisfy $c_j$ we obtain that $z$ encodes a truth assignment that satisfies $c_j$ for all $j \in [1,m]$. Therefore we obtain by
\begin{displaymath}
\sigma(x_i) = \begin{cases}
1 & \text{if } z \equiv 1 \bmod p_i\\
0 & \text{if } z \equiv 0 \bmod p_i
\end{cases}
\end{displaymath}
a satisfying truth assignment $\sigma$ for $C$.

Vice versa suppose $C$ is satisfiable and let $\sigma$ be a satisfying truth assignment. Let $z \in \mathbb{N}$ be the smallest non-negative integer satisfying
\begin{align*}
z &\equiv 1 \bmod 2\\
z &\equiv \begin{cases}
1 \bmod p_i & \text{if } \sigma(x_i)=1\\
0 \bmod p_i & \text{if } \sigma(x_i)=0.
\end{cases}
\end{align*}
Then clearly $l_\infty(\zeta_i,\eta_i^z) \leq k$ by Lemma~\ref{pkinftyleqk}. Now consider $\lambda_j^z$ and $\mu_j$. We have $(k+2)^{\lambda_j^z} = k$ and $(k+2)^{\mu_j} = w_j$ giving us the distance $k-w_j < k$. Moreover we have $k^{\lambda_j^z} = k+2$ and $k^{\mu_j} = k$ with the distance $k+2-k = 2$. Now consider $(\alpha_j\beta_j\gamma_j)^z = \alpha_j^z\beta_j^z\gamma_j^z = \alpha_j^{z_1}\beta_j^{z_2}\gamma_j^{z_3}$ for some $z_1,z_2,z_3 \in \{0,1\}$. By \eqref{equationevaluationjalphabetagamma} there is a $w \in [1,8]$ such that $w^{\alpha_j^{z_1}\beta_j^{z_2}\gamma_j^{z_3}} = 1$. Then we have $w \neq w_j$ because $\sigma$ is a satisfying truth assignment that satisfies $c_j$. Therefore we have $w_j^{\alpha_j^{z_1}\beta_j^{z_2}\gamma_j^{z_3}} \geq 2$ and $w_j^{\mu_j} = k+2$ giving us a distance of $k+2-w_j^{\alpha_j^{z_1}\beta_j^{z_2}\gamma_j^{z_3}} \leq k$. Moreover for all $y \in [1,q_j] \setminus \{w_j\}$ we have $y^{\alpha_j^{z_1}\beta_j^{z_2}\gamma_j^{z_3}} \in [1,q_j]$ and $y^{\mu_j} = y$ giving us a distance of at most $q_j-1 < k$. Finally $\{k+1\} \cup [q_j+1,k-1]$ are fixed-points in both $\lambda_j$ and $\mu_j$. Therefore we obtain $l_\infty(\mu_j,\lambda_j^z) \leq k$ and thus $l_\infty(\tau,\pi^z) \leq k$.
\end{proof}

\subsubsection{Fixed $k$}
\begin{lemma}\label{lemmak1atmost2solutions}
Let $\alpha,\beta \in S_n$ and $\alpha = \alpha_1 \cdots \alpha_d$ be the disjoint cycle decomposition of $\alpha$ and let $a_i$ denote the length of $\alpha_i$. Let $X = \{x \in \mathbb{Z} \mid l_\infty(\beta, \alpha^x) \leq 1\}$. Then for all $i \in [1,d]$ there are at most two numbers $0 \leq y_1,y_2 < a_i$ such that for all $x \in X$ the following holds: $x \equiv y_1 \bmod a_i$ or $x \equiv y_2 \bmod a_i$.
\end{lemma}
\begin{proof}
Let $i \in [1,d]$ and suppose $\alpha_i = (i_1,\dots,i_{a_i})$ where we assume w.l.o.g. $i_1 < i_j$ for all $j \in [2,a_i]$.

Case~1: There exists $1 \leq h \leq a_i$ such that $i_h^\beta = i_1$. Then for all $x \in X$ we have $i_h^{\alpha_i^x} \in \{i_1,i_1+1\}$ which can hold only for at most two different values in $[0,a_i-1]$.

Case~2: For all $1 \leq h \leq a_i$ we have $i_h^\beta \neq i_1$. Then there is a value $e \in [1,n] \setminus \{i_1,\dots,i_{a_i}\}$ such that $e^\beta = i_1$. Hence there is also a value $g \in [1,a_i]$ such that $i_g^\beta = f \not \in \{i_1,\dots,i_{a_i}\}$. Then for all $x \in X$ we have $i_g^{\alpha_i^x} \in \{f-1,f+1\} \cap \{i_1,\dots,i_{a_i}\}$ which can hold only for at most two different values in $[0,a_i-1]$.
\end{proof}

\begin{theorem}
Let $\alpha,\beta \in S_n$ be given in standard representation. Then it can be decided in {\sf NL} whether there is a number $z \in \mathbb{N}$ such that $l_\infty(\beta, \alpha^z) \leq 1$.
\end{theorem}
\begin{proof}
We will give a log-space reduction to \textsc{2-SAT} which is {\sf NL}-complete \cite{papadimitriou} and use the following notations:
\begin{enumerate}
\item $x_1 \Rightarrow x_2$ for $x_1 \lor \neg x_2$
\item $x_1 \xor x_2$ for $(x_1 \lor x_2) \land (\neg x_1 \lor \neg x_2)$.
\end{enumerate}
In the first step we check in log-space for every fixed point $i^\alpha = i$ whether $i \in \{i^\beta-1,i^\beta,i^\beta+1\}$. In the following it therefore suffices to consider cycles of length at least $2.$ Since $\alpha$ is given in standard representation we can compute in log-space the cycle representation of $\alpha$ \cite{cook}. Let $\alpha = \alpha_1 \cdots \alpha_m$ be the disjoint cycle decomposition (without fixed points) of $\alpha$ and let $a_i \geq 2$ denote the length of $\alpha_i$. For $i=1,\dots,m$ we define the ordered set
\begin{displaymath}
X_i = \{v \mid 0 \leq v < a_i, \forall j \in \act(\alpha_i): j^{\alpha_i^v} \in \{j^\beta-1,j^\beta,j^\beta+1\}\}
\end{displaymath}
and $X_{m+1} = \emptyset$. By Lemma~\ref{lemmak1atmost2solutions} we have $|X_i| \leq 2$. When we write $X_i = \{v_1,v_2\}$ we mean $v_1 < v_2$. If there is an $i \in [1,m]$ with $|X_i| = 0$ there clearly is no such $z$. Therefore we assume in the following $1 \leq |X_i| \leq 2$ for all $i \in [1,m]$. When we speak of the $p$-adic valuation of some $a_i$ we always mean the case that $\nu_p(a_i) \geq 1$. For every prime power $p^d \leq n$ with $d \geq 1$ (there clearly are at most $n$ such prime powers) we define $i_{p,d} = \min(\{j \mid d = \nu_p(a_j)\} \cup \{m+1\})$ and define the ordered set
\begin{displaymath}
Y_{p,d} = \{u \in [0,p^d-1] \mid \exists v \in X_{i_{p,d}}: v \equiv u \bmod p^d\}.
\end{displaymath}
Note that we have $0 \leq |Y_{p,d}| \leq 2$. If $|Y_{p,d}| = 0$ then there is no $i \in [1,m]$ with $d = \nu_p(a_i)$. We use $kY_{p,d}$ to denote the $k^\text{th}$ element of $Y_{p,d}$. Now we introduce $|Y_{p,d}|+1$ variables $x_{p,d,0},\dots,x_{p,d,|Y_{p,d}|}$ for all $p^d \leq n$ and define a \textsc{2-SAT} formula by the following:
\begin{displaymath}
F_0 = \bigwedge_{p^d \leq n} \neg x_{p,d,0} \land \bigwedge_{\substack{p^d \leq n, \\|Y_{p,d}| = 2}} (x_{p,d,1} \xor x_{p,d,2}) \land \bigwedge_{\substack{p^d \leq n, \\|Y_{p,d}| = 1}} x_{p,d,1}.
\end{displaymath}
Moreover for every prime $p \leq n$ we define
\begin{displaymath}
F'_p = \bigwedge_{p^d \leq n} \bigwedge_{\substack{p^e \leq n,\\d \leq e}} \bigwedge_{k_1=1}^{|Y_{p,d}|} \bigwedge_{k_2=1}^{|Y_{p,e}|} \varphi(p,d,e,k_1,k_2)
\end{displaymath}
in which we have
\begin{displaymath}
\varphi(p,d,e,k_1,k_2) = \begin{cases}
x_{p,e,k_2} \Rightarrow x_{p,d,k_1} & \text{if } k_1Y_{p,d} \equiv k_2Y_{p,e} \bmod p^d\\
x_{p,e,k_2} \Rightarrow \neg x_{p,d,k_1} & \text{if } k_1Y_{p,d} \not \equiv k_2Y_{p,e} \bmod p^d.
\end{cases}
\end{displaymath}
Now for all $i \in [1,m]$ and every prime power $p^d \mid a_i$ with $d = \nu_p(a_i)$ we define literals by the following: if $X_i = \{v\}$ we define
\begin{displaymath}
\tilde{x}_{i,p,d,0} = \begin{cases}
x_{p,d,1} & \text{if } |Y_{p,d}|=1,1Y_{p,d} \equiv v \bmod p^d\\
x_{p,d,1} & \text{if } |Y_{p,d}|=2,1Y_{p,d} \equiv v \bmod p^d\\
x_{p,d,2} & \text{if } |Y_{p,d}|=2,2Y_{p,d} \equiv v \bmod p^d\\
x_{p,d,0} & \text{otherwise}
\end{cases}
\end{displaymath}
and if $X_i = \{v_1,v_2\}$ we define in the case $v_1 \not \equiv v_2 \bmod p^d$
\begin{displaymath}
\tilde{x}_{i,p,d,1} = \begin{cases}
x_{p,d,1} & \text{if } |Y_{p,d}|=1,1Y_{p,d} \equiv v_1 \bmod p^d\\
x_{p,d,1} & \text{if } |Y_{p,d}|=2,1Y_{p,d} \equiv v_1 \bmod p^d\\
x_{p,d,2} & \text{if } |Y_{p,d}|=2,2Y_{p,d} \equiv v_1 \bmod p^d\\
x_{p,d,0} & \text{otherwise}
\end{cases}
\end{displaymath}
and
\begin{displaymath}
\tilde{x}_{i,p,d,2} = \begin{cases}
x_{p,d,1} & \text{if } |Y_{p,d}|=1,1Y_{p,d} \equiv v_2 \bmod p^d\\
x_{p,d,1} & \text{if } |Y_{p,d}|=2,1Y_{p,d} \equiv v_2 \bmod p^d\\
x_{p,d,2} & \text{if } |Y_{p,d}|=2,2Y_{p,d} \equiv v_2 \bmod p^d\\
x_{p,d,0} & \text{otherwise}.
\end{cases}
\end{displaymath}
If $v_1 \equiv v_2 \bmod p^d$ we define
\begin{displaymath}
\tilde{x}_{i,p,d,0} = \begin{cases}
x_{p,d,1} & \text{if } |Y_{p,d}|=1,1Y_{p,d} \equiv v_1 \bmod p^d\\
x_{p,d,1} & \text{if } |Y_{p,d}|=2,1Y_{p,d} \equiv v_1 \bmod p^d\\
x_{p,d,2} & \text{if } |Y_{p,d}|=2,2Y_{p,d} \equiv v_1 \bmod p^d\\
x_{p,d,0} & \text{otherwise}
\end{cases}
\end{displaymath}
and define the formula
\begin{displaymath}
F_i = \begin{cases}
F_{i,1} & \text{if } |X_i|=1\\
F_{i,2} \land F_{i,3} & \text{if } |X_i|=2
\end{cases}
\end{displaymath}
in which
\begin{displaymath}
F_{i,1} = \bigwedge_{\substack{p^d \mid a_i\\\text{ with } d = \nu_p(a_i)}} \tilde{x}_{i,p,d,0}
\end{displaymath}
and
\begin{displaymath}
F_{i,2} = \bigwedge_{\substack{p^d \mid a_i\\\text{ with } d = \nu_p(a_i),\\v_1 \not \equiv v_2 \bmod p^d}} \bigwedge_{\substack{q^e \mid a_i\\\text{ with } e = \nu_q(a_i),\\v_1 \not \equiv v_2 \bmod q^e}} (\tilde{x}_{i,p,d,1} \xor \tilde{x}_{i,q,e,2})
\end{displaymath}
and
\begin{displaymath}
F_{i,3} = \bigwedge_{\substack{p^d \mid a_i\\\text{ with } d = \nu_p(a_i),\\v_1 \equiv v_2 \bmod p^d}} \tilde{x}_{i,p,d,0}
\end{displaymath}
for all $i \in [1,m]$. Finally we define our \textsc{2-SAT} formula $F$ by
\begin{displaymath}
F = F_0 \land \bigwedge_{i=1}^m F_i \land \bigwedge_{p \leq n} F'_p.
\end{displaymath}
Now we will show there is a number $z \in \mathbb{N}$ such that $l_\infty(\beta, \alpha^z) \leq 1$ if and only if $F$ is satisfiable.

Suppose there is a number $z \in \mathbb{N}$ such that $l_\infty(\beta, \alpha^z) \leq 1$. For all $i \in [1,m]$ let $0 \leq z_i < a_i$ be the smallest positive integer such that $z_i \equiv z \bmod a_i$. Then we have
\begin{displaymath}
\alpha^z = \prod_{i=1}^m \alpha_i^z = \prod_{i=1}^m \alpha_i^{z_i}.
\end{displaymath}
Then clearly $z_i \in X_i$ for all $i \in [1,m]$. Now we define a truth assignment $\sigma$ by the following: for every prime power $p^d \leq n$ with $d \geq 1$ we define
\begin{displaymath}
\sigma(x_{p,d,0}) = 0.
\end{displaymath}
Moreover for all prime powers $p^d$ with $1 \leq |Y_{p,d}| \leq 2$ we define
\begin{displaymath}
\sigma(x_{p,d,1}) = 1
\end{displaymath}
if $|Y_{p,d}| = 1$. In the case $|Y_{p,d}| = 2$ note that we have $z_{i_{p,d}} \in X_{i_{p,d}}$ and hence we either have $1Y_{p,d} \equiv z_{i_{p,d}} \bmod p^d$ or $2Y_{p,d} \equiv z_{i_{p,d}} \bmod p^d$. We define
\begin{displaymath}
\sigma(x_{p,d,1}) = \begin{cases}
1 & \text{if } 1Y_{p,d} \equiv z_{i_{p,d}} \bmod p^d\\
0 & \text{if } 1Y_{p,d} \not \equiv z_{i_{p,d}} \bmod p^d
\end{cases}
\end{displaymath}
and
\begin{displaymath}
\sigma(x_{p,d,2}) = \begin{cases}
0 & \text{if } 2Y_{p,d} \not \equiv z_{i_{p,d}} \bmod p^d\\
1 & \text{if } 2Y_{p,d} \equiv z_{i_{p,d}} \bmod p^d.
\end{cases}
\end{displaymath}
Note that we have $\sigma(x_{p,d,1}) = 1$ if and only if $\sigma(x_{p,d,2}) = 0$. Now we will show that $\sigma$ satisfies $F$.
\begin{claim}\label{claimsignmasatF0}
$\sigma$ satisfies $F_0$.
\end{claim}
We have $\sigma(x_{p,d,0}) = 0$ by definition. Moreover in the case $|Y_{p,d}| = 1$ we have $\sigma(x_{p,d,1}) = 1$ and if $|Y_{p,d}| = 2$ then we have $\sigma(x_{p,d,1}) = 1$ if and only if $\sigma(x_{p,d,2}) = 0$. Thus the subformula $F_0$ clearly evaluates to true.
\qed
\begin{claim}\label{claimsignmasatFp}
$\sigma$ satisfies $F'_p$ for all primes $p \leq n$.
\end{claim}
It suffices to consider the case $\sigma(x_{p,e,k_2}) = 1$. Since $\sigma(x_{p,e,k_2}) = 1$ we have $k_2Y_{p,e} \equiv z_{i_{p,e}} \bmod p^e$. If $|Y_{p,e}| = 1$ this follows from the definition of $Y_{p,e}$ and if $|Y_{p,e}| = 2$ this follows from the definition of $\sigma$. In the case $\varphi(p,d,e,k_1,k_2) = x_{p,e,k_2} \Rightarrow x_{p,d,k_1}$ we have $\sigma(x_{p,d,k_1}) = 1$ if $|Y_{p,d}| = 1$ by definition of $\sigma$ and if $|Y_{p,d}| = 2$ we have
\begin{displaymath}
z_{i_{p,d}} \equiv z_{i_{p,e}} \equiv k_2Y_{p,e} \equiv k_1Y_{p,d} \bmod p^d
\end{displaymath}
and hence $\sigma(x_{p,d,k_1}) = 1$ and $\varphi(p,d,e,k_1,k_2)$ evaluates to true. Now we consider the case $\varphi(p,d,e,k_1,k_2) = x_{p,e,k_2} \Rightarrow \neg x_{p,d,k_1}$. Suppose $|Y_{p,d}| = 1$. Then we have $\sigma(x_{p,d,k_1}) = 1$ by definition. Moreover since $z_{i_{p,d}} \in X_{i_{p,d}}$ we have $k_1Y_{p,d} \equiv z_{i_{p,d}} \bmod p^d$ by definition of $Y_{p,d}$. Then we obtain on the one hand
\begin{displaymath}
k_2Y_{p,e} \equiv z_{i_{p,e}} \equiv z_{i_{p,d}} \equiv k_1Y_{p,d} \bmod p^d
\end{displaymath}
and on the other hand
\begin{displaymath}
k_2Y_{p,e} \not \equiv k_1Y_{p,d} \bmod p^d
\end{displaymath}
by definition of $\varphi(p,d,e,k_1,k_2)$ which is a contradiction. Hence $|Y_{p,d}| = 2$ and we finally obtain
\begin{displaymath}
z_{i_{p,d}} \equiv z_{i_{p,e}} \equiv k_2Y_{p,e} \not \equiv k_1Y_{p,d} \bmod p^d
\end{displaymath}
which gives us $\sigma(x_{p,d,k_1}) = 0$ and $\varphi(p,d,e,k_1,k_2)$ evaluates to true. Note that $z_{i_{p,e}} \equiv z_{i_{p,d}} \bmod p^d$ because $d \leq e$. Thus $F'_p$ evaluates to true.
\qed
\begin{claim}\label{claimsigmasatFi}
$\sigma$ satisfies $F_i$ for all $i \in [1,m]$.
\end{claim}
In the case $X_i = \{v\}$ we have $F_i = F_{i,1}$. Since $z_i \in X_i$ we have $v = z_i$. Moreover we have $z_{i_{p,d}} \equiv z_i \bmod p^d$ for all prime powers $p^d \mid a_i$ with $d = \nu_p(a_i)$. Hence there is a $k \in [1,2]$ such that $v \equiv z_i \equiv z_{i_{p,d}} \equiv kY_{p,d} \bmod p^d$. From this it follows now that $\tilde{x}_{i,p,d,0} = x_{p,d,k}$. If $|Y_{p,d}| = 1$ then $k=1$ and $\sigma(x_{p,d,1}) = 1$ by definition and if $|Y_{p,d}| = 2$ then $z_{i_{p,d}} \equiv kY_{p,d} \bmod p^d$ and hence $\sigma(x_{p,d,k}) = 1$ by definition which satisfies $F_{i,1}$.

In the case $X_i = \{v_1,v_2\}$ we have $F_i = F_{i,2} \land F_{i,3}$. Let $p^d \mid a_i$ be such that $d = \nu_p(a_i)$ and $v_1 \equiv v_2 \bmod p^d$. Since $z_i \in X_i$ we have $z_i \equiv v_1 \equiv v_2 \bmod p^d$. Moreover we have $z_{i_{p,d}} \equiv z_i \bmod p^d$ for all prime powers $p^d \mid a_i$ with $d = \nu_p(a_i)$. Hence there is a $k \in [1,2]$ such that $v_1 \equiv v_2 \equiv z_i \equiv z_{i_{p,d}} \equiv kY_{p,d} \bmod p^d$. From this it follows now that $\tilde{x}_{i,p,d,0} = x_{p,d,k}$. If $|Y_{p,d}| = 1$ then $k=1$ and $\sigma(x_{p,d,1}) = 1$ by definition and if $|Y_{p,d}| = 2$ then $z_{i_{p,d}} \equiv kY_{p,d} \bmod p^d$ and hence $\sigma(x_{p,d,k}) = 1$ by definition which satisfies $F_{i,3}$. Now let $p^d \mid a_i$ be such that $d = \nu_p(a_i)$ and $v_1 \not \equiv v_2 \bmod p^d$ and let $q^e \mid a_i$ be such that $e = \nu_q(a_i)$ and $v_1 \not \equiv v_2 \bmod q^e$. Since $z_i \in X_i$ there is an $l \in [1,2]$ such that $z_i = v_l$. Moreover we have $z_{i_{p,d}} \equiv z_i \bmod p^d$ for all prime powers $p^d \mid a_i$ with $d = \nu_p(a_i)$. Hence there is a $k_1 \in [1,2]$ such that $v_l \equiv z_i \equiv z_{i_{p,d}} \equiv k_1Y_{p,d} \bmod p^d$. Furthermore we have $z_{i_{q,e}} \equiv z_i \bmod q^e$ for all prime powers $q^e \mid a_i$ with $e = \nu_q(a_i)$. Hence there is a $k_2 \in [1,2]$ such that $v_l \equiv z_i \equiv z_{i_{q,e}} \equiv k_2Y_{q,e} \bmod q^e$. We then have
\begin{displaymath}
\tilde{x}_{i,p,d,1} = \begin{cases}
x_{p,d,k_1} & \text{if } l=1\\
x_{p,d,3-k_1} & \text{if } l=2,|Y_{p,d}|=2,v_{3-l} \equiv (3-k_1)Y_{p,d} \bmod p^d\\
x_{p,d,0} & \text{otherwise}
\end{cases}
\end{displaymath}
and
\begin{displaymath}
\tilde{x}_{i,q,e,2} = \begin{cases}
x_{q,e,k_2} & \text{if } l=2\\
x_{q,e,3-k_2} & \text{if } l=1,|Y_{q,e}|=2,v_{3-l} \equiv (3-k_2)Y_{q,e} \bmod q^e\\
x_{q,e,0} & \text{otherwise}.
\end{cases}
\end{displaymath}
By this we obtain one of the following four cases
\begin{displaymath}
\tilde{x}_{i,p,d,1} \xor \tilde{x}_{i,q,e,2} = \begin{cases}
x_{p,d,k_1} \xor x_{q,e,0}\\
x_{p,d,k_1} \xor x_{q,e,3-k_2}\\
x_{p,d,0} \xor x_{q,e,k_2}\\
x_{p,d,3-k_1} \xor x_{q,e,k_2}.
\end{cases}
\end{displaymath}
We have $\sigma(x_{q,e,0}) = 0$ and $\sigma(x_{p,d,k_1}) = 1$ if $|Y_{p,d}| = 1$ and if $|Y_{p,d}| = 2$ we have $\sigma(x_{p,d,k_1}) = 1$ because $z_{i_{p,d}} \equiv k_1Y_{p,d} \bmod p^d$. Thus $x_{p,d,k_1} \xor x_{q,e,0}$ is satisfied. Since we have $z_{i_{q,e}} \equiv k_2Y_{q,e} \bmod q^e$ we clearly have $z_{i_{q,e}} \not \equiv (3-k_2)Y_{q,e} \bmod q^e$ and thus $\sigma(x_{q,e,3-k_2}) = 0$ and $x_{p,d,k_1} \xor x_{q,e,3-k_2}$ is satisfied. Moreover we have $\sigma(x_{p,d,0}) = 0$ and $\sigma(x_{q,e,k_2}) = 1$ if $|Y_{q,e}| = 1$ and if $|Y_{q,e}| = 2$ we have $\sigma(x_{q,e,k_2}) = 1$ because $z_{i_{q,e}} \equiv k_2Y_{q,e} \bmod q^e$. Thus $x_{p,d,0} \xor x_{q,e,k_2}$ is satisfied. Since we have $z_{i_{p,d}} \equiv k_1Y_{p,d} \bmod p^d$ we clearly have $z_{i_{p,d}} \not \equiv (3-k_1)Y_{p,d} \bmod p^d$ and thus $\sigma(x_{p,d,3-k_1}) = 0$ and $x_{p,d,3-k_1} \xor x_{q,e,k_2}$ is satisfied. We finally obtain that $F_i$ is satisfied.
\qed

By Claim~\ref{claimsignmasatF0},\ref{claimsignmasatFp} and \ref{claimsigmasatFi} it follows now that $F$ is satisfied by $\sigma$.

Vice versa suppose $F$ is satisfiable and let $\sigma$ be a satisfying truth assignment. Then for every prime power $p^d$ with $|Y_{p,d}| > 0$ we define numbers $b_{p,d}$ by the following
\begin{displaymath}
b_{p,d} = \begin{cases}
1Y_{p,d} & \text{if } |Y_{p,d}| = 1\\
1Y_{p,d} & \text{if } |Y_{p,d}| = 2, \sigma(x_{p,d,1}) = 1\\
2Y_{p,d} & \text{if } |Y_{p,d}| = 2, \sigma(x_{p,d,2}) = 1.
\end{cases}
\end{displaymath}
Note that by the subformula $F_0$ we have if $|Y_{p,d}| = 1$ then $\sigma(x_{p,d,1}) = 1$ and if $|Y_{p,d}| = 2$ then $x_{p,d,1} \xor x_{p,d,2}$ gives us either $\sigma(x_{p,d,1}) = 1$ or $\sigma(x_{p,d,2}) = 1$. Thus we have $b_{p,d} = kY_{p,d}$ if and only if $\sigma(x_{p,d,k}) = 1$ for some $k \in [1,2]$. For all $i \in [1,m]$ we define the number $b_i$ as the smallest positive integer satisfying the congruences
\begin{displaymath}
b_i \equiv b_{p,d} \bmod p^d
\end{displaymath}
for all prime powers $p^d \mid a_i$ with $d = \nu_p(a_i)$. Then we have $0 \leq b_i < a_i$.
\begin{claim}\label{claimbiinXi}
For all $i \in [1,m]$ we have $b_i \in X_i$.
\end{claim}
In the case $X_i = \{v\}$ we have for every prime power $p^d \mid a_i$ with $d = \nu_p(a_i)$ that $1 = \sigma(\tilde{x}_{i,p,d,0}) = \sigma(x_{p,d,k})$ for some $k \in [1,2]$ by $F_{i,1}$ since $F_0$ gives us $\sigma(x_{p,d,0}) = 0$ and hence $kY_{p,d} \equiv v \bmod p^d$. Thus we obtain $b_{p,d} = kY_{p,d}$ from which it follows now that
\begin{displaymath}
b_i \equiv b_{p,d} \equiv kY_{p,d} \equiv v \bmod p^d.
\end{displaymath}
All congruences together now give us $b_i \equiv v \bmod a_i$ and since $0 \leq b_i,v < a_i$ we obtain $b_i = v$. In the case $X_i = \{v_1,v_2\}$ we have for every prime power $p^d \mid a_i$ with $d = \nu_p(a_i)$ and $v_1 \equiv v_2 \bmod p^d$ that $1 = \sigma(\tilde{x}_{i,p,d,0}) = \sigma(x_{p,d,k})$ for some $k \in [1,2]$ by $F_{i,3}$ and hence $kY_{p,d} \equiv v_1 \equiv v_2 \bmod p^d$. Thus we obtain $b_{p,d} = kY_{p,d}$ from which it follows now that
\begin{displaymath}
b_i \equiv b_{p,d} \equiv kY_{p,d} \equiv v_1 \equiv v_2 \bmod p^d.
\end{displaymath}
Moreover we either have $\sigma(\tilde{x}_{i,p,d,1}) = 1$ and $\sigma(\tilde{x}_{i,q,e,2}) = 0$ or $\sigma(\tilde{x}_{i,p,d,1}) = 0$ and $\sigma(\tilde{x}_{i,q,e,2}) = 1$ for every prime power $p^d \mid a_i$ with $d = \nu_p(a_i)$ and $v_1 \not \equiv v_2 \bmod p^d$ and all $q^e \mid a_i$ with $e = \nu_q(a_i)$ and $v_1 \not \equiv v_2 \bmod q^e$. This follows from the following: let $p_1^{d_1} \mid a_i$ with $d_1 = \nu_{p_1}(a_i)$ and $p_2^{d_2} \mid a_i$ with $d_2 = \nu_{p_2}(a_i)$ be prime powers (we may have $p_1 = p_2$) and assume $\sigma(\tilde{x}_{i,p_1,d_1,1}) = c$ and $\sigma(\tilde{x}_{i,p_2,d_2,2}) = c$ for some $c \in \{0,1\}$. Then $F_{i,2}$ gives us $\tilde{x}_{i,p_1,d_1,1} \xor \tilde{x}_{i,p_2,d_2,2}$ which yields a contradiction. Now let $l \in [1,2]$ be such that $\sigma(\tilde{x}_{i,p,d,l}) = 1$ for every prime power $p^d \mid a_i$ with $d = \nu_p(a_i)$ and $v_1 \not \equiv v_2 \bmod p^d$ and let $k \in [1,2]$ be such that $\tilde{x}_{i,p,d,l} = x_{p,d,k}$. Note that $k=0$ is not possible since $\sigma(\tilde{x}_{i,p,d,l}) = 1$ and $\sigma(x_{p,d,0}) = 0$ by $F_0$. Then we have $v_l \equiv kY_{p,d} \bmod p^d$ and $\sigma(x_{p,d,k}) = 1$ from which it follows now that
\begin{displaymath}
b_i \equiv b_{p,d} \equiv kY_{p,d} \equiv v_l \bmod p^d.
\end{displaymath}
All congruences together now give us $b_i \equiv v_l \bmod a_i$ and since $0 \leq b_i,v_l < a_i$ we obtain $b_i = v_l$.
\qed
\begin{claim}\label{claimbequivbiai}
There is $b \in \mathbb{N}$ such that $b \equiv b_i \bmod a_i$ for all $i \in [1,m]$.
\end{claim}
Let $i \in [1,m]$ and $j \in [1,m]$ be such that $p^d \mid a_i$ is a prime power with $d = \nu_p(a_i)$ and $p^e \mid a_j$ is a prime power with $e = \nu_p(a_j)$ and $d \leq e$. Then there are $k_1,k_2 \in [1,2]$ such that $\sigma(x_{p,d,k_1}) = 1$ and $\sigma(x_{p,e,k_2}) = 1$ because of $F_0$. Then we have $b_{p,d} = k_1Y_{p,d}$ and $b_{p,e} = k_2Y_{p,e}$. Assume $b_{p,d} \not \equiv b_{p,e} \bmod p^d$. Then the subformula $F'_p$ gives us
\begin{displaymath}
\varphi(p,d,e,k_1,k_2) = \begin{cases}
x_{p,e,k_2} \Rightarrow x_{p,d,k_1} & \text{if } k_1Y_{p,d} \equiv k_2Y_{p,e} \bmod p^d\\
x_{p,e,k_2} \Rightarrow \neg x_{p,d,k_1} & \text{if } k_1Y_{p,d} \not \equiv k_2Y_{p,e} \bmod p^d.
\end{cases}
\end{displaymath}
We have
\begin{displaymath}
k_1Y_{p,d} \equiv b_{p,d} \not \equiv b_{p,e} \equiv k_2Y_{p,e} \bmod p^d
\end{displaymath}
by assumption which gives us
\begin{displaymath}
\varphi(p,d,e,k_1,k_2) = x_{p,e,k_2} \Rightarrow \neg x_{p,d,k_1}.
\end{displaymath}
Since we have $\sigma(x_{p,d,k_1}) = 1$ and $\sigma(x_{p,e,k_2}) = 1$ we obtain that $F'_p$ evaluates to false which is a contradiction. Thus $b_{p,d} \equiv b_{p,e} \bmod p^d$ and we can define $b \equiv b_i \bmod a_i$ for all $i \in [1,m]$.
\qed

By Claim~\ref{claimbequivbiai} we can define $0 \leq b < \ord(\alpha)$ as the smallest positive integer satisfying $b \equiv b_i \bmod a_i$ for all $i \in [1,m]$. Then we have
\begin{displaymath}
\alpha^b = \prod_{i=1}^m \alpha_i^b = \prod_{i=1}^m \alpha_i^{b_i}
\end{displaymath}
in which by Claim~\ref{claimbiinXi} we have $b_i \in X_i$ from which it finally follows that for all $j \in [1,n]$ we have
\begin{displaymath}
j^{\alpha^b} \in \{j^\beta-1,j^\beta,j^\beta+1\}
\end{displaymath}
and hence $l_\infty(\beta, \alpha^b) \leq 1$.
\end{proof}

\begin{lemma}\label{linfty31lemma2}
Let $l \geq 3$ be an integer and let $a \in [0,l-1]$. Then we have $l_\infty(\llbracket l \rrbracket^a, \llbracket l \rrbracket^x) \leq 1$ if and only if $x \equiv a \bmod l$.
\end{lemma}
\begin{proof}
One direction follows immediately since we clearly have $l_\infty(\llbracket l \rrbracket^a, \llbracket l \rrbracket^a) = 0$. Now suppose $l_\infty(\llbracket l \rrbracket^a, \llbracket l \rrbracket^x) \leq 1$. It suffices to show $l_\infty(\llbracket l \rrbracket^a, \llbracket l \rrbracket^x) > 1$ if $x \not \equiv a \bmod l$. In the case $b \in [1,l-2]$ we have $(l-a)^{\llbracket l \rrbracket^a} = l$ and $(l-a)^{\llbracket l \rrbracket^{a+b}} = b$ giving us a distance of at least $2$. In the case $b=l-1$ and $a=0$ we have $1^{\llbracket l \rrbracket^0} = 1$ and $1^{\llbracket l \rrbracket^{l-1}} = l$ which gives us a distance of $l-1$. Finally consider the remaining case $b=l-1$ and $a \in [1,l-1]$. We have $(l-a+1)^{\llbracket l \rrbracket^a} = 1$ and $(l-a+1)^{\llbracket l \rrbracket^{a+(l-1)}} = (l-a+1)^{\llbracket l \rrbracket^{a-1}} = l$ which gives us also a distance of $l-1$.
\end{proof}

\begin{theorem}\label{theoremlinftyl1l2}
Let $t \in \mathbb{N}$ be odd and let $0 \leq t_1 < t_2 < t$ be such that $t_1 \not \equiv t_2 \bmod p$ for all primes $p$ with $p \mid t$. Then there is a cycle $\alpha$ of length $t$ and a permutation $\beta$ in which $\beta$ is a product of disjoint $2$-cycles such that $l_\infty(\beta, \alpha^{t_1}) \leq 1$ and $l_\infty(\beta, \alpha^{t_2}) \leq 1$.
\end{theorem}
\begin{proof}
We define
\begin{displaymath}
\omega = t_2-t_1.
\end{displaymath}
Then $\omega$ is a generator of the additive group $(\mathbb{Z}_t,+)$ since $t_1 \not \equiv t_2 \bmod p$ for all primes $p$ with $p \mid t$ and in particular $\omega$ and $t$ are coprime and we can define $0 \leq \psi < t$ as the smallest positive integer satisfying
\begin{displaymath}
\psi \equiv \omega^{-1}(t-t_1) \bmod t
\end{displaymath}
since $\omega^{-1} \bmod t$ exists. For $i=0,\dots,t-1$ we define $0 \leq \omega_i < t$ as the smallest positive integer satisfying $\omega_i \equiv i\omega \bmod t$. Moreover for $i=0,\dots,t-1$ we define $0 \leq \psi_i < t$ as the smallest positive integer satisfying $\psi_i \equiv \psi + i \bmod t$. Now we define the cycle $\alpha = (\alpha_0,\dots,\alpha_{t-1})$ of length $t$ by the following:
\begin{displaymath}
\alpha_{\omega_i} = \begin{cases}
2i+1 & \text{if } 0 \leq i \leq \frac{t-1}{2}\\
2(t-i) & \text{if } \frac{t+1}{2} \leq i \leq t-1.
\end{cases}
\end{displaymath}
For $i=0,\dots,t-1$ we define $0 \leq d_{i,1},d_{i,2} < t$ as the smallest positive integers satisfying
\begin{displaymath}
d_{i,k} \equiv i + t_k \bmod t
\end{displaymath}
for $k \in [1,2]$. Then $\alpha^{t_k}$ maps $\alpha_i$ to $\alpha_{d_{i,k}}$ for all $i \in [0,t-1]$.

\begin{claim}\label{claimreldi1di2}
Let $i \in [0,t-1]$ and let $j \in [0,t-1]$ be such that $\omega_j = d_{i,1}$. If $j=t-1$ then $d_{i,2} = \omega_0$ and if $0 \leq j \leq t-2$ then $d_{i,2} = \omega_{j+1}$.
\end{claim}
We have
\begin{eqnarray*}
& d_{i,1} & \equiv \ \ i + t_1 \bmod t\\
\Leftrightarrow & i & \equiv \ \ d_{i,1} - t_1 \bmod t\\
\Leftrightarrow & d_{i,2} \equiv i + t_2 & \equiv \ \ d_{i,1} - t_1 + t_2 \equiv d_{i,1} + \omega \bmod t\\
\Leftrightarrow & d_{i,2} & \equiv \ \ \omega_j + \omega \bmod t\\
\Leftrightarrow & d_{i,2} & \equiv \ \ j\omega + \omega \bmod t\\
\Leftrightarrow & d_{i,2} & \equiv \ \ (j+1)\omega \bmod t\\
\Leftrightarrow & d_{i,2} & \equiv \ \ \begin{cases}
\omega_0 \bmod t & \text{if } j = t-1\\
\omega_{j+1} \bmod t & \text{if } j \in [0,l-2].
\end{cases}
\end{eqnarray*}
Since $0 \leq d_{i,2} < t$ and $0 \leq \omega_0,\omega_{j+1} < t$ we finally obtain
\begin{displaymath}
d_{i,2} = \begin{cases}
\omega_0 & \text{if } j = t-1\\
\omega_{j+1} & \text{if } j \in [0,t-2].
\end{cases}
\end{displaymath}
\qed

By Claim~\ref{claimreldi1di2} we have $1 \leq |\alpha_{d_{i,1}} - \alpha_{d_{i,2}}| \leq 2$ for all $i \in [0,t-1]$ and if $|\alpha_{d_{i,1}} - \alpha_{d_{i,2}}| = 1$ then either $\alpha_{d_{i,2}} = 1$ or $\alpha_{d_{i,1}} = t$. We define for all $i \in [0,t-1]$
\begin{displaymath}
\alpha_i' = \begin{cases}
\frac{\alpha_{d_{i,1}}+\alpha_{d_{i,2}}}{2} & \text{if } |\alpha_{d_{i,1}} - \alpha_{d_{i,2}}| = 2\\
1 & \text{if } |\alpha_{d_{i,1}} - \alpha_{d_{i,2}}| = 1 \text{ and } \alpha_{d_{i,2}} = 1\\
t & \text{if } |\alpha_{d_{i,1}} - \alpha_{d_{i,2}}| = 1 \text{ and } \alpha_{d_{i,1}} = t.
\end{cases}
\end{displaymath}
Now we define the permutation $\beta$ by the following:
\begin{displaymath}
\beta = \prod_{i=0}^{t-1} \beta_i
\end{displaymath}
in which
\begin{displaymath}
\beta_i = \begin{cases}
(\alpha_i, \alpha_i') & \text{if } \alpha_i < \alpha_i'\\
\mathrm{id} & \text{otherwise}.
\end{cases}
\end{displaymath}

\begin{claim}\label{claimreldomegaidomegai1}
For all $i \in [0,t-1]$ we have $d_{\omega_{\psi_i},1} = \omega_i$.
\end{claim}
We have
\begin{align*}
d_{\omega_{\psi_i},1} &\equiv \omega_{\psi_i} + t_1 \bmod t\\
&\equiv \psi_i \omega + t_1 \bmod t\\
&\equiv ((t-t_1)\omega^{-1}+i)\omega + t_1 \bmod t\\
&\equiv t-t_1+i\omega+t_1 \bmod t\\
&\equiv \omega_i \bmod t.
\end{align*}
Since $0 \leq d_{\omega_{\psi_i},1},\omega_i < t$ we finally obtain $d_{\omega_{\psi_i},1} = \omega_i$.
\qed

\begin{claim}\label{claimrelpsiil1i}
For all $i \in [0,t-1]$ we have $\alpha'_{\omega_{\psi_i}} = \alpha_{\omega_{t-1-i}}$.
\end{claim}
Note that by Claim~\ref{claimreldomegaidomegai1} we have $d_{\omega_{\psi_i},1} = \omega_i$. Then we have by Claim~\ref{claimreldi1di2}
\begin{displaymath}
d_{\omega_{\psi_i},2} = \begin{cases}
\omega_{i+1} & \text{if } i \in [0,t-2]\\
\omega_0 & \text{if } i = t-1.
\end{cases}
\end{displaymath}
Case~1: $i \in [0,\frac{t-3}{2}]$. We have
\begin{align*}
\alpha'_{\omega_{\psi_i}}&= \frac{\alpha_{d_{\omega_{\psi_i},1}} + \alpha_{d_{\omega_{\psi_i},2}}}{2}\\
&= \frac{\alpha_{\omega_i} + \alpha_{\omega_{i+1}}}{2}\\
&= \frac{2i+1 + 2(i+1)+1}{2}\\
&= 2i+2\\
&= 2(t-(t-1-i))\\
&= \alpha_{\omega_{t-1-i}}.
\end{align*}
Case~2: $i = \frac{t-1}{2}$. We have $\alpha_{d_{\omega_{\psi_i},1}} = \alpha_{\omega_i} = t$ and hence
\begin{displaymath}
\alpha'_{\omega_{\psi_i}} = \alpha_{d_{\omega_{\psi_i},1}} = t = \alpha_{\omega_{t-1-\frac{t-1}{2}}}.
\end{displaymath}
Case~3: $i \in [\frac{t+1}{2},t-2]$. We have
\begin{align*}
\alpha'_{\omega_{\psi_i}}&= \frac{\alpha_{d_{\omega_{\psi_i},1}} + \alpha_{d_{\omega_{\psi_i},2}}}{2}\\
&= \frac{\alpha_{\omega_i} + \alpha_{\omega_{i+1}}}{2}\\
&= \frac{2(t-i) + 2(t-(i+1))}{2}\\
&= 2(t-i)-1\\
&= 2(t-1-i)+1\\
&= \alpha_{\omega_{t-1-i}}.
\end{align*}
Case~4: $i=t-1$. We have $\alpha_{d_{\omega_{\psi_i},2}} = \alpha_{\omega_0} = 1$ and hence
\begin{displaymath}
\alpha'_{\omega_{\psi_i}} = \alpha_{d_{\omega_{\psi_i},2}} = 1 = \alpha_{\omega_{t-1-(t-1)}}.
\end{displaymath}
\qed

\begin{claim}\label{claimrelalphaialphaj}
For all $i \in [0,t-1]$ and $j \in [0,t-1]$ we have $\alpha_i' = \alpha_j$ if and only if $\alpha_j' = \alpha_i$.
\end{claim}
Suppose $\alpha_i' = \alpha_j$ and let $0 \leq e < t$ be such that $i = \omega_{\psi_e}$. Note that $e$ exists. Since $\omega$ is a generator there is $c \in [0,t-1]$ such that $i \equiv \omega_c \bmod t$ and we can choose $e \equiv c - \psi \bmod t$. By Claim~\ref{claimrelpsiil1i} we have $\alpha'_{\omega_{\psi_e}} = \alpha_{\omega_{t-1-e}}$ (i.e. $j = \omega_{t-1-e}$). Claim~\ref{claimrelpsiil1i} also gives us
\begin{align*}
\alpha'_j &= \alpha'_{\omega_{t-1-e}}\\
&= \alpha'_{\omega_{\psi-\psi+t-1-e}}\\
&= \begin{cases}
\alpha'_{\omega_{\psi_{t-1-e-\psi}}} & \text{if } \psi \leq t-1-e\\
\alpha'_{\omega_{\psi_{t-1-e+t-\psi}}} & \text{if } \psi > t-1-e
\end{cases}\\
&= \begin{cases}
\alpha_{\omega_{t-1-(t-1-e-\psi)}} & \text{if } \psi \leq t-1-e\\
\alpha_{\omega_{t-1-(t-1-e+t-\psi)}} & \text{if } \psi > t-1-e
\end{cases}\\
&= \begin{cases}
\alpha_{\omega_{e+\psi}} & \text{if } \psi \leq t-1-e\\
\alpha_{\omega_{e-t+\psi}} & \text{if } \psi > t-1-e
\end{cases}\\
&= \alpha_{\omega_{\psi_e}}\\
&= \alpha_i.
\end{align*}
Note that $0 \leq \psi_e < t$ and hence $\psi_e = e + \psi$ if $\psi \leq t-1-e$ because $e + \psi \leq t-1$ and $\psi_e = e - t + \psi$ if $\psi > t-1-e$ because $e+\psi \geq t$.
\qed

By Claim~\ref{claimrelalphaialphaj} it also follows that the cycles of $\beta$ are disjoint: suppose $\beta$ contains $(\alpha_h,\alpha_i)$ and $(\alpha_j,\alpha_i)$ for some pairwise different $h,i,j \in [0,t-1]$. Then we have $\alpha_h = \alpha_i' = \alpha_j$ which yields $h = j$ contradicting $h \neq j$.

Now it suffices to show $l_\infty(\beta, \alpha^{t_1}) \leq 1$ and $l_\infty(\beta, \alpha^{t_2}) \leq 1$. By Claim~\ref{claimreldi1di2} we have for all $i \in [0,t-1]$ and all $k \in [1,2]$
\begin{displaymath}
\alpha_i^{\alpha^{t_k}} \in \{\alpha_j-1,\alpha_j,\alpha_j+1\}
\end{displaymath}
for some $j \in [0,t-1]$. Moreover we have $\alpha_i' = \alpha_j$. By Claim~\ref{claimrelalphaialphaj} we then have $\alpha_j' = \alpha_i$ and hence
\begin{displaymath}
\alpha_j^{\alpha^{t_k}} \in \{\alpha_i-1,\alpha_i,\alpha_i+1\}.
\end{displaymath}
Thus we either have $\beta_i = (\alpha_i, \alpha_j)$ and $\beta_j = \mathrm{id}$ or $\beta_i = \mathrm{id}$ and $\beta_j = (\alpha_i, \alpha_j)$ yielding a distance of at most $1$ if $i \neq j$. In the case
\begin{displaymath}
\alpha_i^{\alpha^{t_k}} \in \{\alpha_i-1,\alpha_i,\alpha_i+1\}
\end{displaymath}
we have that $\alpha_i$ is a fixed point in $\beta$ yielding a distance of at most $1$. Thus we finally obtain $l_\infty(\beta, \alpha^{t_1}) \leq 1$ and $l_\infty(\beta, \alpha^{t_2}) \leq 1$.
\end{proof}

\begin{corollary}\label{corollarylinftyl1l2}
Let $t \in \mathbb{N}$ be odd and let $0 \leq t_1 < t_2 < t$ be such that $t_1 \not \equiv t_2 \bmod p$ for all primes $p$ with $p \mid t$. Moreover let $d \geq 3$ be an integer with $\gcd(d,t) = 1$ and let $0 \leq d_0 < d$. Then there are permutations $\gamma, \delta \in S_{t+d}$ such that $l_\infty(\delta, \gamma^{a_1}) \leq 1$ and $l_\infty(\delta, \gamma^{a_2}) \leq 1$ in which $a_r$ satisfies the congruences $a_r \equiv d_0 \bmod d$ and $a_r \equiv t_r \bmod t$ for $r \in [1,2]$.
\end{corollary}
\begin{proof}
Let $\alpha$ and $\beta$ be the permutations that Theorem~\ref{theoremlinftyl1l2} yields regarding the numbers $t_1,t_2,t$. Moreover let $\varepsilon = \llbracket d \rrbracket$. We define $\gamma = (\alpha,\varepsilon) \in S_t \times S_d$ and $\delta = (\beta,\varepsilon^{d_0}) \in S_t \times S_d$. Then we have $l_\infty(\delta,\gamma^{a_1}) \leq 1$ and $l_\infty(\delta,\gamma^{a_2}) \leq 1$ since $\gamma^{a_r} = (\alpha^{a_r},\varepsilon^{a_r}) = (\alpha^{t_r},\varepsilon^{d_0})$ with $r \in [1,2]$ and clearly $l_\infty(\varepsilon^{d_0},\varepsilon^{d_0}) \leq 1$ and $l_\infty(\beta,\alpha^{t_r}) \leq 1$ follows from Theorem~\ref{theoremlinftyl1l2}.
\end{proof}

\begin{theorem}\label{npcompletelinfty}
The {\sc Subgroup distance problem} regarding the $l_\infty$ distance is {\sf NP}-complete when the input group is abelian and given by two generators and $k=1$.
\end{theorem}
\begin{proof}
We give a log-space reduction from \textsf{X3HS}. Let $X$ be a finite set and $\mathcal{B} \subseteq 2^X$ be a set of subsets of $X$ all of size $3$. W.l.o.g. assume that $X = [1,n]$ and let $\mathcal{B} = \{C_1,\dots,C_m\}$. For $i \in X$ we denote by $D_i \subseteq [1,m]$ the ordered set of all numbers $j$ such that $i \in C_j$. For $k \in [1,|D_i|]$ we denote by $kD_i$ the $k^\text{th}$ element of $D_i$. For $i \in [1,n]$ and $j \in [0,m]$ let $p_{i,j}$ be the $(jn+i)^\text{th}$ odd prime. We define $q_j = \prod_{i \in C_j} p_{i,j}$. Moreover let $N = \sum_{i=1}^n p_{i,0}p_{i,m}|D_i| + 2\sum_{j=1}^m (p_{n,j}^2 + p_{n,j})$. We will work with the group
\begin{displaymath}
G = \prod_{i=1}^n V_i \times \prod_{j=1}^m U_j
\end{displaymath}
with $V_i = S_{p_{i,0}p_{i,m}}^{|D_i|}$ and $U_j = S_{p_{n,j}^2+p_{n,j}}^2$ which naturally embeds into $S_N$. We define auxiliary permutations by the following: for $i \in [1,n]$ and $k \in [1,|D_i|]$ let $\alpha_{i,kD_i}$ and $\beta_{i,kD_i}$ be the permutations that Theorem~\ref{theoremlinftyl1l2} yields regarding the solutions $0 \leq x_{i,k,1},x_{i,k,2} < p_{i,0}p_{i,kD_i}$ with
\begin{align*}
x_{i,k,1}&\equiv 0 \bmod p_{i,0}\\
x_{i,k,1}&\equiv 0 \bmod p_{i,kD_i}
\end{align*}
and
\begin{align*}
x_{i,k,2}&\equiv 1 \bmod p_{i,0}\\
x_{i,k,2}&\equiv 1 \bmod p_{i,kD_i}
\end{align*}
in which $\alpha_{i,kD_i}$ is a cycle of length $p_{i,0}p_{i,kD_i}$ and $\beta_{i,kD_i}$ is a product of $2$-cycles. Moreover we define the following: for $j \in [1,m]$ let $\gamma_{j,1},\delta_{j,1}$ be the permutations that Corollary~\ref{corollarylinftyl1l2} yields regarding the solutions $0 \leq y_{j,1},y_{j,2} < q_j$ with
\begin{align*}
y_{j,1}&\equiv 1 \bmod p_{i_1,j}\\
y_{j,1}&\equiv 0 \bmod p_{i_2,j}\\
y_{j,1}&\equiv 0 \bmod p_{i_3,j}
\end{align*}
and
\begin{align*}
y_{j,2}&\equiv 0 \bmod p_{i_1,j}\\
y_{j,2}&\equiv 1 \bmod p_{i_2,j}\\
y_{j,2}&\equiv 0 \bmod p_{i_3,j}
\end{align*}
in which $i_1 < i_2 < i_3 \in C_j$ are the elements of $C_j$. Furthermore for $j \in [1,m]$ let $\gamma_{j,2},\delta_{j,2}$ be the permutations that Corollary~\ref{corollarylinftyl1l2} yields regarding the solutions $0 \leq z_{j,1},z_{j,2} < q_j$ with
\begin{align*}
z_{j,1}&\equiv 0 \bmod p_{i_1,j}\\
z_{j,1}&\equiv 0 \bmod p_{i_2,j}\\
z_{j,1}&\equiv 0 \bmod p_{i_3,j}
\end{align*}
and
\begin{align*}
z_{j,2}&\equiv 1 \bmod p_{i_1,j}\\
z_{j,2}&\equiv 0 \bmod p_{i_2,j}\\
z_{j,2}&\equiv -1 \bmod p_{i_3,j}
\end{align*}
in which $i_1 < i_2 < i_3 \in C_j$ are the elements of $C_j$. Note that these permutations can be constructed in log-space. Also note that these solutions are the only solutions by Lemma~\ref{lemmak1atmost2solutions} and Lemma~\ref{linfty31lemma2}. We define the input group elements $\tau, \pi_1, \pi_2 \in G$ as follows where $i$ ranges over $[1,n]$, $j$ ranges over $[1,m]$ and $k$ ranges over $[1,|D_i|]$:
\begin{align*}
\tau &= (\tau_1,\dots,\tau_n, \tau'_1,\dots,\tau'_m) \text{ with}\\
\tau_i &= (\tau_{i,1},\dots,\tau_{i,|D_i|})\\
\tau_{i,k} &= \beta_{i,kD_i}\\
\tau'_j &= (\delta_{j,1},\delta_{j,2})
\end{align*}
\begin{align*}
\pi_1 &= (\rho_{1,1},\dots,\rho_{1,n},\sigma_{1,1},\dots,\sigma_{1,m}) \text{ with}\\
\rho_{1,i} &= (\rho_{1,i,1},\dots,\rho_{1,i,|D_i|})\\
\rho_{1,i,k} &= \alpha_{i,kD_i}\\
\sigma_{1,j} &= (\gamma_{j,1},\mathrm{id})
\end{align*}
and
\begin{align*}
\pi_2 &= (\rho_{2,1},\dots,\rho_{2,n},\sigma_{2,1},\dots,\sigma_{2,m}) \text{ with}\\
\rho_{2,i} &= (\rho_{2,i,1},\dots,\rho_{2,i,|D_i|})\\
\rho_{2,i,k} &= \mathrm{id}\\
\sigma_{2,j} &= (\gamma_{j,1},\gamma_{j,2}).
\end{align*}
Note that $\pi_1$ and $\pi_2$ commute.

Now we will show there are $x_1,x_2 \in \mathbb{N}$ such that $l_\infty(\tau, \pi_1^{x_1}\pi_2^{x_2}) \leq 1$ if and only if there is a subset $X' \subseteq X$ such that $|X' \cap C_j| = 1$ for all $j \in [1,m]$.

Suppose there are $x_1,x_2 \in \mathbb{N}$ such that $l_\infty(\tau, \pi_1^{x_1}\pi_2^{x_2}) \leq 1$. Then we define
\begin{displaymath}
X' = \{i \mid x_1 \equiv 1 \bmod p_{i,0}\}.
\end{displaymath}
\begin{claim}\label{claimlinftyikDi}
For all $i \in [1,n]$ and all $k \in [1,|D_i|]$ the following holds: $x_1 \equiv 0,1 \bmod p_{i,kD_i}$ and $x_1 \equiv 0,1 \bmod p_{i,0}$. Moreover $x_1 \equiv 1 \bmod p_{i,0}$ if and only if $x_1 \equiv 1 \bmod p_{i,kD_i}$.
\end{claim}
The first part follows from the fact that $l_\infty(\tau_{i,k}, \rho_{1,i,k}^{x_1}\rho_{2,i,k}^{x_2}) = l_\infty(\beta_{i,kD_i}, \alpha_{i,kD_i}^{x_1}) \leq 1$ if and only if $x_1 \equiv x_{i,k,1},x_{i,k,2} \bmod p_{i,0}p_{i,kD_i}$. The second part follows from the definitions of $x_{i,k,1}$ and $x_{i,k,2}$.
\qed

\begin{claim}\label{claimlinftyjexactly1a}
For all $j \in [1,m]$ there is exactly one $a \in C_j$ such that $x_1 \equiv 1 \bmod p_{a,j}$ and $x_1 \equiv 0 \bmod p_{b,j}$ for all $b \in C_j \setminus \{a\}$.
\end{claim}
Consider the projection onto the factor $U_j$. We have $l_\infty(\tau'_j, \sigma_{1,j}^{x_1}\sigma_{2,j}^{x_2}) \leq 1$ which gives us the two statements
\begin{equation}\label{eqclaimujsplit1}
l_\infty(\delta_{j,1}, \gamma_{j,1}^{x_1}\gamma_{j,1}^{x_2}) \leq 1
\end{equation}
and
\begin{equation}\label{eqclaimujsplit2}
l_\infty(\delta_{j,2}, \gamma_{j,2}^{x_2}) \leq 1.
\end{equation}
By \eqref{eqclaimujsplit2} we obtain $x_2 \equiv z_{j,1},z_{j,2} \bmod q_j$. Moreover $l_\infty(\delta_{j,1}, \gamma_{j,1}^x) \leq 1$ holds if and only if $x \equiv y_{j,1},y_{j,2} \bmod q_j$. Hence by \eqref{eqclaimujsplit1} we obtain $x_1 + x_2 \equiv y_{j,1},y_{j,2} \bmod q_j$. If $x_2 \equiv z_{j,1} \equiv 0 \bmod q_j$ we obtain $x_1 \equiv y_{j,1},y_{j,2} \bmod q_j$. If $x_2 \equiv z_{j,2} \bmod q_j$ we obtain the following
\begin{align*}
x_1 + x_2 &\equiv x_1 + 1 \bmod p_{i_1,j}\\
x_1 + x_2 &\equiv x_1 + 0 \bmod p_{i_2,j}\\
x_1 + x_2 &\equiv x_1 - 1 \bmod p_{i_3,j}
\end{align*}
in which $i_1 < i_2 < i_3 \in C_j$ are the elements of $C_j$. In the case $x_1+x_2 \equiv y_{j,2} \bmod q_j$ we obtain
\begin{align*}
x_1 + 1 &\equiv 0 \bmod p_{i_1,j}\\
x_1 + 0 &\equiv 1 \bmod p_{i_2,j}\\
x_1 - 1 &\equiv 0 \bmod p_{i_3,j}
\end{align*}
which gives us by
\begin{align*}
x_1 &\equiv -1 \bmod p_{i_1,j}\\
x_1 &\equiv 1 \bmod p_{i_2,j}\\
x_1 &\equiv 1 \bmod p_{i_3,j}
\end{align*}
a contradiction since $x_1 \equiv -1 \bmod p_{i_1,j}$ is not possible by Claim~\ref{claimlinftyikDi}. For this also note that $p_{i,j} \geq 3$. Thus $x_1+x_2 \equiv y_{j,1} \bmod q_j$ and
\begin{align*}
x_1 + 1 &\equiv 1 \bmod p_{i_1,j}\\
x_1 + 0 &\equiv 0 \bmod p_{i_2,j}\\
x_1 - 1 &\equiv 0 \bmod p_{i_3,j}
\end{align*}
which gives us
\begin{align*}
x_1 &\equiv 0 \bmod p_{i_1,j}\\
x_1 &\equiv 0 \bmod p_{i_2,j}\\
x_1 &\equiv 1 \bmod p_{i_3,j}.
\end{align*}
Now we define $0 \leq y_{j,3} < q_j$ as the smallest positive integer satisfying the congruences
\begin{align*}
y_{j,3} &\equiv 0 \bmod p_{i_1,j}\\
y_{j,3} &\equiv 0 \bmod p_{i_2,j}\\
y_{j,3} &\equiv 1 \bmod p_{i_3,j}.
\end{align*}
Hence there is exactly one $a \in [1,3]$ such that $x_1 \equiv y_{j,a} \equiv 1 \bmod p_{i_a,j}$ and for all $b \in [1,3] \setminus \{a\}$ we have $x_1 \equiv y_{j,a} \equiv 0 \bmod p_{i_b,j}$ which proves the claim. For this also note that the congruence for $x_2$ can be chosen suitably for every $j \in [1,m]$. This does not interfere other congruences since $q_{j_1}$ and $q_{j_2}$ are coprime for $j_1 \neq j_2$.
\qed

Now we show $|X' \cap C_j| = 1$ for all $j \in [1,m]$. Let $j \in [1,m]$. By Claim~\ref{claimlinftyjexactly1a} there is exactly one $a \in C_j$ such that $x_1 \equiv 1 \bmod p_{a,j}$ and $x_1 \equiv 0 \bmod p_{b,j}$ for all $b \in C_j \setminus \{a\}$. By Claim~\ref{claimlinftyikDi} we have $x_1 \equiv 1 \bmod p_{i,a}$ if and only if $x_1 \equiv 1 \bmod p_{a,0}$. Hence $a \in X'$. Moreover by Claim~\ref{claimlinftyikDi} $x_1 \equiv 0 \bmod p_{b,j}$ if and only if $x_1 \equiv 0 \bmod p_{b,0}$ and by this $b \not \in X'$ for all $b \in C_j \setminus\{a\}$. Hence $|X' \cap C_j| = 1$.

Vice versa suppose there is a subset $X' \subseteq X$ such that $|X' \cap C_j| = 1$ for all $j \in [1,m]$. Then we define $x_1 \in \mathbb{N}$ as the smallest positive integer satisfying the congruences
\begin{displaymath}
x_1 \equiv \begin{cases}
1 \bmod p_{i,0}p_{i,kD_i} & \text{if } i \in X'\\
0 \bmod p_{i,0}p_{i,kD_i} & \text{if } i \not \in X'
\end{cases}
\end{displaymath}
for all $i \in [1,n]$ and all $k \in [1,|D_i|]$. Then by projecting onto the factor $V_i$ we clearly have $l_\infty(\tau_{i,k}, \rho_{1,i,k}^{x_1}\rho_{2,i,k}^{x_2}) = l_\infty(\beta_{i,k}, \alpha_{i,kD_i}^{x_1}) \leq 1$ because $x_1 \equiv x_{i,k,1} \bmod p_{i,0}p_{i,kD_i}$ or $x_1 \equiv x_{i,k,2} \bmod p_{i,0}p_{i,kD_i}$. Since $|X' \cap C_j| = 1$ for all $j \in [1,m]$ there is exactly one $a \in C_j$ such that $x_1 \equiv 1 \bmod p_{a,j}$ and $x_1 \equiv 0 \bmod p_{b,j}$ for all $b \in C_j \setminus \{a\}$ and thus $x_1 \equiv y_{j,a} \bmod q_j$. Hence we can define $x_2 \in \mathbb{N}$ as the smallest positive integer satisfying the congruences
\begin{displaymath}
x_2 \equiv \begin{cases}
z_{j,1} \bmod q_j & \text{if } x_1 \equiv y_{j,1},y_{j,2} \bmod q_j\\
z_{j,2} \bmod q_j & \text{if } x_1 \equiv y_{j,3} \bmod q_j.
\end{cases}
\end{displaymath}
Then by projecting onto the factor $U_j$ we have
\begin{displaymath}
l_\infty(\delta_{j,1}, \gamma_{j,1}^{x_1}\gamma_{j,1}^{x_2}) \leq 1
\end{displaymath}
and
\begin{displaymath}
l_\infty(\delta_{j,2}, \gamma_{j,2}^{x_2}) \leq 1
\end{displaymath}
because $x_2 \equiv z_{j,1} \bmod q_j$ or $x_2 \equiv z_{j,2} \bmod q_j$ and $x_1 + x_2 \equiv y_{j,3} + z_{j,2} \equiv y_{j,1} \bmod q_j$ or $x_1 + x_2 \equiv x_1 + z_{j,1} \equiv x_1 \equiv y_{j,1},y_{j,2} \bmod q_j$ which gives us $l_\infty(\tau'_j, \sigma_{1,j}^{x_1}\sigma_{2,j}^{x_2}) \leq 1$ from which it follows now that $l_\infty(\tau, \pi_1^{x_1}\pi_2^{x_2}) \leq 1$.
\end{proof}

\section{Conclusion}
We have shown that the \textsc{Subgroup distance problem} is {\sf NP}-complete in cyclic permutation groups for all metrics mentioned in the introduction. This paper only focuses on the \textsc{Subgroup distance problem} but in the literature also the maximum subgroup distance problem was studied in \cite{buchheim} and the weight problem and further variants were studied in \cite{cameron}. Further research could try to show also for these problems {\sf NP}-completeness when the input group is cyclic or at least given by a few generators since {\sf NP}-completeness is not necessarily obtainable for cyclic groups. Although the \textsc{Subgroup distance problem} is {\sf NP}-complete in cyclic permutation groups for all metrics mentioned in the introduction this does not necessarily hold for the minimum weight problem in cyclic groups. We give an example: consider the minimum weight problem regarding the Hamming weight (i.e. $w_H(\tau) = |\{i \mid i^\tau \neq i\}|$). It can be decided in polynomial time whether there is a number $z \in \mathbb{N}$ with $z \not \equiv 0 \bmod \ord(\tau)$ such that $w_H(\tau^z) \leq k$ for some $\tau \in S_n$ by simply checking whether there is a prime $p \mid \ord(\tau)$ such that $w_H\left(\tau^{\frac{\ord(\tau)}{p}}\right) \leq k$. Note that such primes are relatively small since $\ord(\tau) \mid n!$ and hence $p \leq n$. On the other hand in \cite{lohrey} it was shown that it is {\sf NP}-complete to decide whether for some given $\alpha,\beta \in S_n$ the coset $\beta \langle \alpha \rangle$ contains a fixed-point-free element $\beta\alpha^z$ for some $z \in \mathbb{N}$. This problem is equivalent to asking whether there is $z \in \mathbb{N}$ such that $H(\beta,\alpha^{-z}) \geq n$. This is seen as follows: for all $i \in [1,n]$ we have $i^{\beta\alpha^z} \neq i$ if and only if for all $i \in [1,n]$ we have $i^{\beta} = i^{\beta\alpha^z\alpha^{-z}} \neq i^{\alpha^{-z}}$. By this the maximum subgroup distance problem regarding the Hamming distance is {\sf NP}-complete when the input group is cyclic.

\subsection{Open Problems}
We have shown that it can be decided in {\sf NL} whether for given permutations $\alpha,\beta \in S_n$ there is $x \in \mathbb{N}$ such that $l_\infty(\beta,\alpha^x) \leq 1$. We do not know if this problem is {\sf NL}-complete or can even be solved in deterministic log-space. Moreover this problem becomes {\sf NP}-complete when the input group is abelian and given by at least $2$ generators. However we were only able to proof {\sf NP}-completeness for the problem $l_\infty(\beta,\alpha^x) \leq k$ when $k$ is part of the input rather than a fixed value. Therefore it remains open whether the \textsc{Subgroup distance problem} regarding the $l_\infty$ distance is {\sf NP}-complete in cyclic permutation groups for any fixed $k \geq 2$.

\paragraph{Acknowledgement}~\bigskip

This work has been supported by the DFG research project Lo748/12-2.

\bibliographystyle{abbrv}

\begin{thebibliography}{21}

\bibitem{arvind2}
V. Arvind.
\newblock The parameterized complexity of some permutation group problems.
\newblock arXiv:1301.0379, 2013.

\bibitem{arvind}
V. Arvind and P. S. Joglekar.
\newblock Algorithmic Problems for Metrics on Permutation Groups.
\newblock In {\em SOFSEM 2008: Theory and Practice of Computer Science}. Lecture Notes in Computer Science, volume 4910, pages 136--147, Springer, Berlin, Heidelberg, 2008.

\bibitem{babai}
L. Babai, E. M. Luks, and {\'A}. Seress.
\newblock Permutation groups in NC.
\newblock In {\em Proceedings of the 19th Annual ACM Symposium on Theory of Computing, STOC 1987}, pages 409--420, ACM, 1987.

\bibitem{buchheim}
C. Buchheim, P. J. Cameron and T. Wu.
\newblock On the subgroup distance problem.
\newblock {\em Discrete Mathematics}, 309(4): 962--968, 2009.

\bibitem{buchheimjunger}
C. Buchheim and M. J\"unger.
\newblock Linear optimization over permutation groups.
\newblock {\em Discrete Optimization}, 2(4): 308--319, 2005.

\bibitem{cameron}
P. J. Cameron and T. Wu.
\newblock The complexity of the weight problem for permutation and matrix groups.
\newblock {\em Discrete Mathematics}, 310(3): 408--416, 2010.

\bibitem{cook}
S. A. Cook and P. McKenzie.
\newblock Problems complete for deterministic logarithmic space.
\newblock {\em Journal of Algorithms}, 8(3): 385--394, 1987.

\bibitem{deza}
M. Deza and T. Huang.
\newblock Metrics on permutations: A survey.
\newblock {\em Journal of Combinatorics, Information \& System Sciences}, 23: 173--185, 1998.

\bibitem{diaconis}
P. Diaconis.
\newblock Group representations in probability and statistics.
\newblock {\em Institute of Mathematical Statistics}, 1988.

\bibitem{dusart}
P. Dusart.
\newblock The $k^{\text{th}}$ prime is greater than $k(\ln k+\ln\ln k-1)$ for $k \geq 2$.
\newblock {\em Mathematics of Computation}, 68(225): 411--415, 1999.

\bibitem{furst}
M. L. Furst, J. E. Hopcroft, and E. M. Luks.
\newblock Polynomial-time algorithms for permutation groups.
\newblock In {\em Proceedings of the 21st Annual Symposium on Foundations of Computer Science, FOCS 1980}, pages 36--41, IEEE Computer Society, 19080.

\bibitem{gareyjohnson}
M. R. Garey and D. S. Johnson.
\newblock Computers and Intractability: A Guide to the Theory of NP-completeness.
\newblock {\em Freeman}, 1979.

\bibitem{lohrey}
M. Lohrey and A. Rosowski.
\newblock Finding cycle types in permutation groups with few generators.
\newblock In {\em Computing and Combinatorics, COCOON 2025.} volume 15984 of LNCS, pages 367--380, Springer, Singapore, 2025.

\bibitem{lucchini}
A. Lucchini, F. Menegazzo, and M. Morigi.
\newblock Generating permutation groups.
\newblock {\em Communications in Algebra}, 32(5): 1729--1746, 2004.

\bibitem{onur}
C. B. Onur.
\newblock A Zero-Knowledge Proof of Knowledge for Subgroup Distance Problem.
\newblock arXiv:2408.00395, 2024.

\bibitem{papadimitriou}
C. H. Papadimitriou.
\newblock Computational complexity.
\newblock {\em Addison-Wesley}, 1994.

\bibitem{pinch}
R. G. E. Pinch.
\newblock {\em The distance of a permutation from a subgroup of $S_n$}.
\newblock Combinatorics and Probability, Cambridge University Press, pages 473--479, 2007.

\bibitem{rosowski}
A. Rosowski.
\newblock On the Subgroup Distance Problem in Cyclic Permutation Groups.
\newblock arXiv:2504.06844v1, 2025.

\bibitem{rosser}
J. B. Rosser.
\newblock Explicit bounds for some functions of prime numbers.
\newblock {\em American Journal of Mathematics}, 63(1): 211--232, 1941.

\bibitem{seress}
{\'A}. Seress.
\newblock {\em Permutation Group Algorithms}.
\newblock Cambridge Tracts in Mathematics. Cambridge University Press, 2003.

\bibitem{sims}
C. C. Sims.
\newblock Computational methods in the study of permutation groups.
\newblock In {\em Computational Problems in Abstract Algebra}, pages 169--183, Pergamon, 1970.

\end{thebibliography}

\end{document}